\documentclass{amsart}

\usepackage{amsmath,amssymb,amsthm,mathrsfs}
\usepackage[active]{srcltx}
\usepackage{graphicx,color}

\numberwithin{equation}{section}

\newtheorem{thm}{Theorem}[section]
\newtheorem{prop}[thm]{Proposition}
\newtheorem{lem}[thm]{Lemma}

\theoremstyle{definition}

\newtheorem{remark}[thm]{Remark}

\newcommand{\sH}{{\mathcal H}}

\newcommand{\sM}{{\mathcal M}}

\newcommand{\sD}{{\mathcal D}}

\newcommand{\sX}{{\mathcal X}}

\def\de{\Delta}

\def\g{\gamma}
\def\ga{\Gamma}

\def\la{\Lambda}

\def\om{\Omega}

\def\tht{\Theta}

\def\ts{\times}

\def\iy{\infty}
\def\im{{\rm Im\, }}

\def\kr{{\rm Ker\, }}
\def\diag{{\rm diag\, }}

\def\BC{{\mathbb C}}
\def\BD{{\mathbb D}}

\def\BM{{\mathbb M}}

\def\BT{{\mathbb T}}

\newcommand{\mat}[2]{\ensuremath{\left[\begin{array}{#1}#2\end{array} \right]}}

\newcommand{\ands}{\quad\mbox{and}\quad}
\newcommand{\wtil}{\widetilde}
\newcommand{\half}{\frac{1}{2}}

\newcommand{\Up}{\Upsilon}

\newcommand{\fR}{\mathfrak{R}}

\newcommand{\nn}{\notag}

\renewcommand{\theequation}{\arabic{section}.\arabic{equation}}

\begin{document}

\author{A.E. Frazho}

\address{%
Department of Aeronautics and Astronautics, Purdue University\\
West Lafayette, IN 47907, USA}

\email{frazho@ecn.purdue.edu}


\author{S. ter Horst}

\address{%
Unit for BMI, North-West University\\
Private Bag X6001-209, Potchefstroom 2520, South Africa}

\email{sanne.terhorst@nwu.ac.za}


\author{M.A. Kaashoek}

\address{%
Department of Mathematics,
VU University Amsterdam\\
De Boelelaan 1081a, 1081 HV Amsterdam, The Netherlands}

\email{m.a.kaashoek@vu.nl}

\thanks{The research of the first author was partially supported by a visitors grant from NWO (Netherlands Organization for Scientific Research).}




\title[State space solutions for a suboptimal rational Leech problem II]
{State space formulas for a suboptimal rational Leech problem II: Parametrization of all solutions}

\date{}

\maketitle

\begin{abstract}
For  the strictly positive case (the suboptimal case), given stable rational matrix functions $G$ and $K$, the set of all $H^\iy$ solutions $X$ to the Leech problem associated with $G$ and $K$, that is, $G(z)X(z)=K(z)$ and $\sup_{|z|\leq 1}\|X(z)\|\leq 1$, is presented as the range of a linear fractional representation of which  the coefficients are presented  in state space form. The matrices involved in the  realizations are computed from state space realizations of the data functions $G$ and $K$. On the one hand the results are based  on the commutant lifting theorem and on the other hand on  stabilizing solutions of algebraic Riccati equations related to  spectral factorizations.
\end{abstract}

\setcounter{section}{0}
\section{Introduction}\label{intro}\setcounter{equation}{0}

The present paper is a continuation of  the paper \cite{FtHK-IEOT}.  As in \cite{FtHK-IEOT} we have given two stable rational matrix functions $G$ and $K$ of sizes  $m\ts p$ and $m\ts q$, respectively,  and   we are interested in   $p\times q$  matrix-valued $H^\iy$  solutions $X$ to the Leech problem:
\begin{equation}\label{Leech1}
G(z)X(z)=K(z) \quad  (|z|<1),\quad \|X\|_\iy=\sup_{|z|<1}\|X(z)\|\leq 1.
\end{equation}
Here \emph{stable} means that  the poles of the functions belong to the set $|z|>1$, infinity included.  In particular, the given functions $G$ and $K$ (as well as  the unknown function $X$) are matrix-valued $H^\infty$ functions.

As  is well-known, a result by R.W. Leech dating from the early seventies, see \cite{Lch}  (and \cite{LchCom}), tells us that for arbitrary  matrix-valued $H^\iy$  functions $G$ and $K$, not necessarily rational,  the problem \eqref{Leech1} is solvable if and only if  the operator $T_GT_G^*-T_KT_K^*$  is nonnegative. Here
\[
T_G:\ell^2_+(\BC^p)\to\ell^2_+(\BC^m) \ands T_K:\ell^2_+(\BC^q)\to\ell^2_+(\BC^m)
\]
are the  (block) Toeplitz operators defined by $G$ and $K$ respectively. Since then it has been shown by various authors that the Leech problem can been solved by using general methods for dealing with metric constrained completion and interpolation problems, including commutant lifting; see the review \cite {LchCom} and the references therein.

In the present paper, as in \cite{FtHK-IEOT}, we deal with the \emph{suboptimal} case where the operator
\begin{equation}
\label{poscond1a}T_GT_G^*-T_KT_K^* \ \mbox{is strictly positive}.
\end{equation}
Note that an $H^\iy$ solution to the Leech problem \eqref{Leech1} exists if and only if the operator $T_GT_G^*-T_KT_K^*$ is positive, see \cite{Lch}. In \cite{FtHK-IEOT}, using  commutant lifting theory and state space methods from mathematical system theory, we proved that the maximum entropy solution to the Leech problem \eqref{Leech1} with rational data is  a stable rational matrix function and we computed a state space formula for this solution. The focus of the current paper is on computing all solutions.

In a few recent publications \cite{T13,tH13,FtHK14}, a different approach to the Leech problem was presented, also leading to state space formulas for a solution. Although it is not hard to modify this approach to compute a set of rational matrix solutions, it remains unclear at this stage if the method is suitable to compute the set of all solutions, cf., \cite{FtHK14-MTNS}.

One of the additional complications in describing the set of all solutions in our approach is that it requires an explicit description of the value at zero  $\tht_0$  of the inner  function $\tht$  associated with the  model  space $\im T_G^*$. Another difficulty,  which already appears in \cite{FtHK-IEOT},  is the fact that the intertwining contraction $\la=T_G^*(T_GT_G^*)^{-1}T_K$  appearing in the commutant lifting setting of the Leech problem is a rather complicated operator. If $K\not =0$ this operator  is not finite dimensional as in the classical Nevanlinna-Pick interpolation problem or a compact operator as in the Nehari problem for the Wiener class but, in general, $\la$ is an infinite dimensional operator which can be Fredholm or invertible (cf., Proposition  \ref{P:coronaCLT} at the end of the present paper).

Before stating our main result, we need some preliminaries. As in \cite{FtHK-IEOT},  the starting point is the fact, well known from mathematical systems theory, that rational matrix functions admit finite dimensional state space realizations. We shall assume that the stable rational matrix function $\begin{bmatrix} G&K \end{bmatrix}$ is given in realized form:
\begin{equation}\label{GKreal}
\mat{cc}{G(z)& K(z)}=\mat{cc}{D_1&D_2} + z C(I_n-zA)^{-1}\mat{cc}{B_1&B_2}.
\end{equation}
Here $I_n$ is the $n\ts n$ identity matrix and $A$, $B_1$, $B_2$, $C$, $D_1$ and $D_2$ are matrices of appropriate size. Without loss of generality we may assume $A$ is a stable matrix, i.e., all eigenvalues of $A$ are in the open unit disc $\BD$, and the pair $\{C,A\}$ is observable. The latter means that $CA^\nu x=0$ for $\nu=0,1,2,\ldots$ implies $x$ is the zero vector in $\BC^n$.
For $j=1,2$ let $P_j$ be the controllability gramians associated with the pair $\{A,B_j\}$, i.e., $P_j$ is the unique solution to the Stein equation
\begin{align}\label{Stein}
P_j-AP_j A^*=B_jB_j^*.
\end{align}

As Theorem 1.1   in \cite{FtHK-IEOT} shows,  since $G$ and $K$ are rational matrix $H^\iy$ functions,  it is possible to present a solution criterion for the Leech problem in terms of matrices derived from the matrices appearing in the realization \eqref{GKreal}. This criterion involves an algebraic Riccati equation that appears in the spectral factorization of the rational $m\ts m$ matrix function
\begin{align}\label{defR}
R(z)=G(z)G^*(z)-K(z)K^*(z).
\end{align}
Here $G^*(z)=G(\bar{z}^{-1})^*$ and $K^*(z)=K(\bar{z}^{-1})^*$. It was computed in \cite{FtHK14} that $R$ admits the state space realization
\begin{align*}
R(z) = zC(I - zA)^{-1}\ga + R_0 + \ga^*(zI - A^*)^{-1}C^*,
\end{align*}
with $R_0$ and $\ga$ the matrices given by
\begin{align}
R_0 &= D_1D_1^* -D_2D_2^* +C(P_1 - P_2)C^*,\label{defR0}\\[.1cm]
 \ga &= B_1D_1^* -B_2D_2^* + A(P_1 - P_2)C^*.  \label{defGa}
\end{align}
Under the hypothesis that $T_GT_G^*-T_KT_K^*$ is strictly positive, the Toeplitz operator $T_R$ defined by $R$ is also strictly positive. The latter is equivalent, see Remark 1.3 in \cite{FtHK-IEOT}, to the existence of a stabilizing solution $Q$ to the algebraic Riccati equation
\begin{equation}\label{are}
Q  =  A^* Q  A + ( C - \Gamma^* Q  A )^*
( R_0 - \Gamma^* Q  \Gamma )^{-1} ( C - \Gamma^* Q  A ).
\end{equation}
In this context, for  the solution $Q$ to  \eqref{are} to be stabilizing   means that the matrix $R_0-\ga^* Q \ga$ must be strictly positive and that the matrix
\begin{align}\label{A0Delta}
A_0= A-\ga\de^{-1}(C-\ga^*QA),\quad\mbox{ with }\quad \de=R_0-\ga^*Q\ga,
\end{align}
must be stable. These two stability conditions  guarantee  that there exists just one stabilizing solution $Q$ to \eqref{are}. Furthermore, since the pair $\{C,A\}$ is observable, the stabilizing solution $Q$ is invertible, cf., \cite[Eq. (1.18)]{FtHK-IEOT}. Theorem 1.1 in \cite{FtHK-IEOT} now states that $T_GT_G^*-T_KT_K^*$ is strictly positive if and only if there exists a stabilizing solution $Q$ to \eqref{are} such that
\begin{align*}
Q^{-1}+P_2-P_1\mbox{ is strictly positive.}
\end{align*}

To state our  main theorem  we need to consider an additional algebraic  Riccati equation.  Note that $T_GT_G^*\geq T_GT_G^*-T_KT_K^*$. Since $T_GT_G^*-T_KT_K^*$ is strictly positive, it follows  that  the same holds true for $T_GT_G^*$.  This allows us to apply the results of the previous paragraph with the function $K$ identically equal to zero, and with $B_2=0$ and $D_2=0$. This leads  to a second  algebraic Riccati equation:
\begin{equation}\label{are2}
Q_0  =  A^* Q_0  A + ( C - \Gamma_0^* Q_0  A )^*
( R_{10} - \Gamma_0^* Q_0  \Gamma_0 )^{-1} ( C - \Gamma_0^* Q_0  A ).
\end{equation}
Here
\begin{equation*}
R_{10} = D_1D_1^* +CP_1C^*,\qquad \ga_0 = B_1D_1^* + AP_1C^*.
\end{equation*}
Since  $T_G$ is right invertible and  the pair $\{C,A\}$ is observable, it follows that \eqref{are2} has a unique stabilizing solution $Q_0$ such that $Q_0^{-1}-P_1$ is strictly positive.

Finally, since $T_GT_G^*$ is strictly positive,  the projection on $\kr T_G=\ell^2_+(\BC^p)\ominus \im T_G^*$ is given by $P_{\kr T_G}=I_p-T_G^*(T_GT_G^*)^{-1}T_G=T_\tht T_\tht^*$, with $\tht$ the inner function associated with the model space $\im T_G^*$. This yields that the value $\tht_0$ of $\tht$ at zero  is uniquely determined, up to a constant unitary matrix of order $p-m$ on the right,   by
\begin{equation} \label{deftht0}
 \tht_0\tht_0^*=I_p-E_p^*T_G^*(T_GT_G^*)^{-1}T_GE_p.
\end{equation}
 Here, for any positive integer $k$, we write $E_k $ for the canonical embedding of $\BC^k$ onto the first coordinate space of $\ell_+^2(\BC^k)$, see \eqref{defEk} below. The fact that the number of columns of  $\tht_0$ is $p-m$ is explained in Remark \ref{rem22} below.  Since the realization $G(z)=D_1+zC(I_n-zA)^{-1}B_1$ is a  stable state space realization, we can apply Theorem 1.1 in  \cite{FKR12} to derive a formula for $\tht_0$ in terms of  the matrices $A$, $B_1$, $C$, $D_1$ and related matrices. Therefore in what follows we shall assume $\tht_0$  is given. We shall refer to $\tht_0$ as the \emph{left minimal rank factor  determined by \eqref{deftht0}}. See Lemma \ref{L:thtids}  in the next section for some further insight in the role of  $\tht_0$.

We are now ready to state our   main theorem which  provides  a characterization of all solutions to the suboptimal rational Leech problem \eqref{Leech1} in the form of the range of a linear fractional transformation.
\begin{thm}\label{T:main1}
Let $G$ and $K$ be stable rational matrix functions of sizes  $m\ts p$ and $m\ts q$, respectively, such that $T_GT_G^*-T_KT_K^*$ is strictly positive, and assume that there is  no non-zero $x\in \BC^p$ such that $G(z)x$ is identically zero on the open unit disc $\BD$. Let $\begin{bmatrix}G&K \end{bmatrix}$ be given by the observable stable realization \eqref{GKreal}. Then the set of solutions to the Leech problem \eqref{Leech1} appears as the range of the linear fractional transformation $Y\mapsto X$ given by
\begin{align}\label{LFTsols}
X(z)=(\Up_{12}(z)+\Up_{11}(z)Y(z))(\Up_{22}(z)+\Up_{21}(z)Y(z))^{-1}.
\end{align}
Here the free parameter $Y$ is any $(p-m)\ts q$ matrix-valued $H^\iy$  function such that $\|Y\|_\iy\leq 1$, and
\begin{align}
\Up_{11}(z)&={\tht_0}\Delta_1^{-1} -zC_1(I-z A_0)^{-1}Q^{-1}(Q^{-1}+P_2-P_1)^{-1}B_1\tht_0 \Delta_1^{-1},\notag\\
\Up_{21} (z) &=-zC_2(I-zA_0)^{-1}Q^{-1}(Q^{-1}+P_2-P_1)^{-1}B_1{\tht_0}\Delta_1^{-1},\nn \\
\Up_{12}(z) &=(D_1^*\de^{-1}D_2+D_1^*C_0\om C_2^*+B_1^*QB_0)\de_0^{-1}+\label{UpsFin}\\
&\hspace{5cm}+ zC_1(I-z A_0)^{-1}B_0 \Delta_0^{-1};\notag\\
\Up_{22} (z) &=\de_0+zC_2(I-zA_0)^{-1}B_0\Delta_0^{-1},\notag
\end{align}
where $A_0$ and $\de$ are given by \eqref{A0Delta}, the matrix $\tht_0$ is the left minimal rank factor determined by \eqref{deftht0}, the matrices $C_j$, $j=0,1,2$, and $B_0$  are given by
\begin{align*}
&C_0=\de^{-1}(C-\ga^* Q A),\quad
C_j=D_j^*C_0+B_j^*QA_0,\ j=1,2,\\
&\hspace{1.6cm}B_0=B_2-\ga \de^{-1}D_2+A_0\om C_2^*,
\end{align*}
with  $\om=(P_1-P_2)(Q^{-1}+P_2-P_1)^{-1}Q^{-1}$, where $Q$   is the stabilizing solution of the Riccati equation \eqref{are}, and $\de_0$ and $\de_1$ are the positive definite matrices determined by
\begin{equation}\label{DeltasFin}
\begin{aligned}
\de_0^2&=I_q+C_2\Omega C_2^*+(D_2-\Gamma^*QB_2)^*\Delta^{-1}(D_2-\Gamma^*QB_2)+B_2^*QB_2,\\
\de_1^2 &=I_{p-m}+{\tht_0}^*B_1^*\big((Q^{-1}+P_2-P_1)^{-1}-(Q_0^{-1}-P_1)^{-1}\big)B_1\tht_0,
\end{aligned}
\end{equation}
where $Q_0$ is the stabilizing solution of the Riccati equation \eqref{are2}.
\end{thm}

\begin{remark} The functions $\Up_{12}$ and $\Up_{22}$  already appear in \cite{FtHK-IEOT}. More precisely, $\Up_{12}(z)\de_0$ is the function $U(z)$ given by \cite[Eq. (5.14)]{FtHK-IEOT}, and $\Up_{22}(z)\de_0$ is the function $V(z)$ given by \cite[Eq. (5.13)]{FtHK-IEOT}. Note that $\Up_{12}(z)\Up_{22}(z)^{-1}=U(z)V(z)^{-1}$ is the solution which one obtains if the free parameter $Y=0$; this solution is  the   maximum entropy solution given by \cite[Eq. (1.12)]{FtHK-IEOT}. Finally,
the coefficient matrix
\[
\Up=\begin{bmatrix}\Up_{11}&\Up_{12}\\  \Up_{21}&\Up_{22} \end{bmatrix}
\]
has a number of interesting properties which follow from the general theory derived in Section \ref{S:InfDim}. For instance,   $\Up$ is $J_1, J_2$-inner, where $J_1=\diag(I_p, -I_q)$, and $J_2=\diag(I_{p-m}, -I_q)$.
\end{remark}

\begin{remark}  All solutions can also be obtained as the range of a  linear fractional map of Redheffer type:
\[
X(z)=\Phi_{22}(z)+\Phi_{21}(z)Y(z)\big(I-\Phi_{11}(z)Y(z)\big)^{-1}\Phi_{12}(z),
\]
where, as in Theorem \ref{T:main1}, the free parameter $Y$ is any $(p-m)\ts q$ matrix-valued $H^\iy$  function such that $\|Y\|_\iy\leq 1$, and the functions $\Phi_{11}$, $\Phi_{12}$, $\Phi_{21}$ and $\Phi_{22}$ are stable rational matrix functions given by stable state space realizations.  In fact, as expected, these coefficients  are uniquely determined by the identities
\begin{align*}
&\Phi_{11}=-\Phi_{12}\Up_{21},\quad
\Phi_{12}=\Up_{22}^{-1},\\
&\Phi_{21}=\Up_{11}-\Up_{12}\Phi_{12}\Up_{21},\quad
\Phi_{22}=\Up_{12}\Phi_{12}.
\end{align*}
We omit further details.
\end{remark}

\begin{remark}
In terms of the realization  \eqref{GKreal}  the condition  that there is  no non-zero $x\in \BC^p$ such that $G(z)x$ is identically zero  on $\BD$  is equivalent to the requirement that   $ \kr \begin{bmatrix} B_1 & D_1 \end{bmatrix}{}^\top$ consists of the zero vector only.  To see this note that $G(z)x=D_1x+zC(I_n-zA)^{-1}B_1 x$. Hence
\[
G(z)x=0 \ (z\in \BD) \Leftrightarrow\  D_1x= 0 \ \mbox{and} \  CA^\nu B_1 x=0 \ (\nu=0,1,2, \ldots).
\]
Since the pair $\{C,A\}$ is observable, it follows that
\[
G(z)x=0 \ (z\in \BD) \Leftrightarrow \  D_1x= 0   \ \mbox{and} \  B_1x=0  \Leftrightarrow\ x\in \kr \begin{bmatrix} B_1 \\ D_1 \end{bmatrix},
\]
which yields the desired result. The condition  that there is  no non-zero $x\in \BC^p$ such that $G(z)x$ is identically zero  on $\BD$ can also be understood as a minimality condition on some isometric liftings; see Lemma \ref{lemmin1} in the next section.
\end{remark}

\noindent
The  paper consists of  five sections.  The first is the present introduction.  Section \ref{S:SurToep} has a preliminary character. In this section $G$ is an arbitrary matrix-valued  $H^\iy$ function, not necessarily rational. Among others we  present the inner function $\tht$  describing the  null space of $T_G$.  In Section \ref{S:InfDim} the functions $G$ and $K$ are  again just matrix-valued $H^\iy$ functions,  not necessarily rational. We derive infinite dimensional state space formulas for the two linear fractional representations of the set of all solutions to the sub-optimal Leech equation, starting  from the abstract commutant lifting results in Section VI.6 of \cite{FFGK98}. In Section \ref{S:SSF} we prove  Theorem  \ref{T:main1}. The final section, Section \ref{App},  has   the character of an appendix; in this section we present a version of the commutant lifting theorem, based on Theorem VI.6.1 in \cite{FFGK98}.  Theorem  \ref{cltstrict2}, which follows Theorem VI.6.1 in \cite{FFGK98} but does not appear in \cite{FFGK98},  serves as the abstract basis  for the proofs of our main results.

\medskip
\paragraph{Notation and terminology.} We conclude this introduction with some notation and terminology used throughout the paper. As usual, we identify a $k\ts r$ matrix with complex entries with the linear operator from $\BC^r$ to $\BC^k$ induced  by the action of the matrix on the standard bases. For any positive integer $k$ we write $E_k $ for the canonical embedding of $\BC^k$ onto the first coordinate space of $\ell_+^2(\BC^k)$, that is,
\begin{equation}
\label{defEk}
E_k  = \begin{bmatrix}
           I_k & 0 & 0 & 0 & \cdots \,\,\\
         \end{bmatrix}{}^\top:\mathbb{C}^k \rightarrow \ell_+^2(\mathbb{C}^k).
\end{equation}
Here $\ell^2_+(\BC^k)$ denotes the Hilbert space of unilateral square summable sequences of vectors in $\BC^k$. By $S_k$ we denote  the unilateral shift on $\ell_+^2(\mathbb{C}^k)$. For positive integers $k$ and $r$ we write $H_{k\ts r}^\iy$ for the Banach space of  all $k\ts r$ matrices with entries from $H^\iy$, the algebra of all bounded analytic functions of the open unit disc $\BD$.  The  supremum norm  of   $F\in H_{k\ts r}^\iy$ is given by  $\|F\|_\iy=\sup_{|z|<1} \|F(z)\|$. By $\fR H^\iy_{k\ts r}$ we denote the space of all stable rational $k\ts r$ matrix functions which we view as  a subspace of   $H_{k\ts r}^\iy$.  The adjoint of  $F\in H_{k\ts r}^\iy$  is the co-analytic function $F^*$ which is defined by $F^*(z)=F(1/\bar{z})^*$, $|z|<1$. Finally, we write $\bigvee_{i\in I} \sM_i $ for the closure of the linear hull of the spaces $\sM_i$ ranging over the index set $I$.


\section{The model space and model operator associated with the kernel of a surjective analytic Toeplitz operator}\label{S:SurToep}

Throughout this section let $G\in H^\iy_{m\ts p}$. Then  $S_m T_G=T_G S_p $  implies $\kr T_G$ is invariant under $S_p$, and hence $\sH'=\im T_G^*=\ell^2_+(\BC^p)\ominus \kr T_G$ is invariant under $S_p^*$. By the Beurling-Lax theorem, $\sH'$ is a model space, that is, there exists an inner function $\tht\in H^\iy_{p\ts k}$, for some $k\leq p$, such that $\sH'=\ell^2_+(\BC^p)\ominus T_\tht \ell^2_+(\BC^k)$. We write $T'$ for the associated model operator $T'=P_{\sH'}S_p|_{\sH'}$.

We shall assume in addition that $T_G$ is a surjective operator, or equivalently, that $T_GT_G^*$ is an invertible operator on $\ell^2_+(\BC^m)$. In that case, we provide an explicit infinite dimensional state space representation for the inner function $\tht$, along with some formulas that will be of use in the sequel.

Note that $S_p$ is an isometric lifting of $T'$, see the appendix for the definition of a (minimal) isometric lifting. In a  second result in this section, Lemma \ref{lemmin1} below, we present a condition which is equivalent to $S_p$ being a minimal isometric lifting of $T'$.

\begin{lem}\label{L:thtids}
The inner function $\tht\in H^\iy_{p\ts k}$ with $\sH'=\ell^2_+(\BC^p)\ominus T_\tht \ell^2_+(\BC^k)$ is given by
\begin{equation}\label{thtform}
\tht(z)={\tht_0} -zE_p^*T_G^*(I-zS_m^*)^{-1}(T_GT_G^*)^{-1}N.
\end{equation}
Here  $N$ is the operator from $\BC^k$ to $\ell^2_+(\BC^m)$ given by $N=S_m^*T_GE_p{\tht_0}$, and $\tht_0=\tht(0)$ is a  one-to-one $p\ts k$ matrix  uniquely determined, up to multiplication with a constant unitary $k\ts k$ matrix from the right, by
\begin{equation}\label{tht0}
{\tht_0}{\tht_0}^*=I_p- E_p^* T_G^*(T_GT_G^*)^{-1}T_GE_p.
\end{equation}
Furthermore, $N=-T_G S_p^* T_\tht E_k$ and for any $z\in\BD$ we have
\begin{equation}\label{thtid}
\tht(z)N^*(T_GT_G^*)^{-1}=E_p^*(I-zS_p^*)^{-1} T_G^*(T_GT_G^*)^{-1} (I-zS_m^*)S_m.
\end{equation}
\end{lem}

\begin{remark}\label{rem22} Note that ${\tht_0}$ is the analog of the  left minimal rank factor  introduced in the second paragraph preceding Theorem \ref{T:main1}. In the rational case $k=p-m$; see Lemma 2.2 in \cite{FKR11}. However, it can be shown that the latter equality  holds in general; see \cite[Section 2]{GtHKprep}.
\end{remark}

\begin{proof}[\bf Proof of Lemma \ref{L:thtids}.]
We first show that $N=-T_G S_p T_\tht E_k$ holds. Using the fact that $T_G T_\Theta=0$ and ${\tht_0}=E_p^* T_\Theta E_k$ we obtain that
\begin{align*}
N
&=S_m^*T_GE_p{\tht_0}
=S_m^*T_GE_pE_p^* T_\Theta E_k
=S_m^*T_G(I- S_pS_p^*)T_\Theta E_k\\
&=-S_m^*T_G S_pS_p^* T_\Theta E_k
=-S_m^*S_m T_G S_p^* T_\Theta E_k
=-T_G S_p^* T_\Theta E_k,
\end{align*}
as claimed.

Since $T_G$ is surjective, $T_G^*(T_GT_G^*)^{-1}T_G$ is the orthogonal projection onto  $\im T_G^*$, so that
\begin{equation}\label{OrthProject}
 T_\tht T_\tht^*=P_{\kr T_G}=I-P_{\im T_G^*}
=I-T_G^*(T_GT_G^*)^{-1}T_G.
\end{equation}
Next observe that
\[
T_{\tht}^* S_p^*T_\tht E_k= S_k^* T_\tht^*T_\tht E_k=S_k^*E_k=0.
\]
Together with the formula for $N$ we then obtain for each $z\in\BD$ that
\begin{align*}
\tht(z)
&=E_p^*(I-zS_p^*)^{-1}T_\tht E_k\\
&=E_p^* T_\Theta E_k+zE_p^*(I-zS_p^*)^{-1}S_p^*T_\tht E_k\\
&={\tht_0}+zE_p^*(I-zS_p^*)^{-1}(I- T_\tht T_\tht^*) S_p^*T_\tht E_k\\
&={\tht_0}+zE_p^*(I-zS_p^*)^{-1}T_G^*(T_GT_G^*)^{-1}T_G S_p^*T_\tht E_k\\
&= {\tht_0} - zE_p^*(I-zS_p^*)^{-1}T_G^*(T_GT_G^*)^{-1}N.
\end{align*}
This yields the desired state space representation  \eqref{thtform}  for $\tht$.

Note that
\[
\kr {\tht_0}\subset \kr S_m^*T_GE_p{\tht_0} = \kr N.
\]
Thus, for $u\in \kr {\tht_0}$, we have $\tht(z)u=0$ for all $z\in\BD$, and hence also for a.e.\ $z\in\BT$. Since $\tht$ is inner, this implies $u=0$. Hence $\kr {\tht_0}=\{0\}$.

Furthermore, since $E_p^* T_\tht S_k=E_p^* S_p T_\tht=0$, we have
\begin{align*}
E_p^* T_\tht T_\tht^* E_p
&=E_p^* T_\tht (E_kE_k^*+S_kS_k^*) T_\tht^* E_p
=E_p^* T_\tht E_kE_k^* T_\tht^* E_p={\tht_0}{\tht_0}^*.
\end{align*}
Along with \eqref{OrthProject}, this yields
\begin{align*}
{\tht_0}{\tht_0}^*=E_p^* (I-T_G^*(T_GT_G^*)^{-1}T_G) E_p=I_p- E_p^*T_G^*(T_GT_G^*)^{-1}T_GE_p.
\end{align*}

Again using $T_GT_\tht=0$ and $N=-T_G S_p^*T_\tht E_k$, we obtain that
\begin{align*}
N E_k^*
&=-T_G S_p^*T_\tht E_k E_k^*=-T_G S_p^*T_\tht(I-S_kS_k^*)\\
&=-T_G S_p^*T_\tht+T_G S_p^*T_\tht S_kS_k^*
=-T_G S_p^*T_\tht+T_G T_\tht S_k^*=-T_G S_p^*T_\tht.
\end{align*}
Fix $z\in \BD$. Then we find
\begin{align*}
\tht(z)N^*
&=E_p^*(I-z S_p^*)^{-1}T_\tht E_kN^*=-E_p^*(I-z S_p^*)^{-1}T_\tht T_\tht^*S_p T_G^*.
\end{align*}
Using \eqref{OrthProject}, yields
\begin{align*}
T_\tht T_\tht^*S_p T_G^* (T_GT_G^*)^{-1}
&=S_p T_G^*(T_GT_G^*)^{-1} -T_G^*(T_GT_G^*)^{-1}T_G S_p T_G^* (T_GT_G^*)^{-1}\\
&=S_p T_G^*(T_GT_G^*)^{-1} -T_G^*(T_GT_G^*)^{-1}S_m.
\end{align*}
Combining this with the formula for $\tht(z)N^*$ gives
\begin{align*}
\tht(z)N^*(T_GT_G^*)^{-1}
&=E_p^*(I-z S_p^*)^{-1}(T_G^*(T_GT_G^*)^{-1}S_m-S_p T_G^*(T_GT_G^*)^{-1})\\
&=E_p^*(I-z S_p^*)^{-1}\times\\
&\qquad\times\big(T_G^*(T_GT_G^*)^{-1}S_m-z T_G^*(T_GT_G^*)^{-1}S_m^* S_m\big)\\
&=E_p^*(I-z S_p^*)^{-1}T_G^*(T_GT_G^*)^{-1}(I-zS_m^*)S_m.
\end{align*}
Hence the identity \eqref{thtid} holds.
\end{proof}

We now proceed with the second result of this section.

\begin{lem}\label{lemmin1}
The shift $S_p$ is a minimal isometric lifting of  $T'=P_{\sH'}S_p|_{\sH'}$  if and only there is no non-zero $x\in \BC^p$ such that $G(z)x$ vanishes identically,  that is, ${\cap}_{z\in\BD}\kr G(z)=\{0\}$.
\end{lem}

\begin{proof}[\bf Proof.]
Put
\[
\sX=\bigvee_{\nu\geq 0} S_p^\nu \sH', \quad \sX_0=\sX\ominus \sH', \quad \sX_1=\ell^2_+(\BC^p)\ominus \sX.
\]
Since $\sX$ is invariant under both $S_p$ and $S_p^*$, the same holds true for $\sX_1$. Hence $S_p$  partitions as
\begin{equation}\label{lifting2a}
S_p=\begin{bmatrix}
   T'   &    0&0\\
     W_0 &  Z_0&0\\\
     0&0&Z_1
\end{bmatrix}\mbox{ on } \begin{bmatrix}
  \sH'\\
  \sX_0\\\sX_1
\end{bmatrix}
\end{equation}
and the isometry
\begin{equation}\label{lifting2b}
U_0'=\begin{bmatrix}
   T'   &    0\\
   W_0 &  Z_0
\end{bmatrix}\mbox{ on } \begin{bmatrix}
  \sH'\\
\sX_0
\end{bmatrix}
\end{equation}
is a  minimal isometric lifting of $T'$.
In particular, the shift $S_p$ is a minimal isometric lifting
of $T'$ if and only $\sX_1$ consists of the zero element only.

Now take
$h=(h_0,h_1,\ldots)\in\ell^2_+(\BC^p)$. Then $h\in \sX_1$ if
and only if $h\perp S^\nu \im T_G^*$ for $\nu=0, 1,2, \ldots$.
In other words
\begin{align*}
 h\in \sX_1& \ \Longleftrightarrow \ T_G(S_p^*)^\nu h=0, \quad \nu=0, 1,2, \ldots\\
& \  \Longleftrightarrow \
 \begin{bmatrix}
G_0  &  0  &0&\cdots\\
G_1 & G_0&0&\cdots\\
 G_2&G_1&G_0&\\
 \vdots&\vdots&&\ddots
\end{bmatrix}
 \begin{bmatrix}h_\nu\\ h_{\nu+1}\\h_{\nu+2}\\ \vdots  \end{bmatrix}= 0, \quad \nu=0, 1,2, \ldots\\
 & \  \Longleftrightarrow \
 \begin{bmatrix}
G_0 \\
G_1\\
 G_2\\
 \vdots
\end{bmatrix}h_\nu,  \quad \nu=0, 1,2, \ldots\\
& \  \Longleftrightarrow \  G(z)h_\nu \equiv 0,   \quad \nu=0, 1,2, \ldots.
\end{align*}
We conclude that $\sX_1$ contains a non-zero element if and only if there exists a non-zero $x\in \BC^p$ such that $G(z)x$ vanishes identically.
\end{proof}


\section{Infinite dimensional  state space formulas   for the coefficients}\label{S:InfDim}

In this section $G\in H^\iy_{m\ts p}$ and $K\in H^\iy_{m\ts q}$,  and we assume that  $T_GT_G^*-T_KT_K^*$ is strictly positive. We do not require $G$ and $K$ to be  rational matrix functions.  Our aim is to describe all solutions to the Leech problem  \eqref{Leech1}.

Note that $T_GT_G^*-T_KT_K^*$ strictly positive  implies that $T_GT_G^*$ is strictly positive, and thus that $T_G$ is a surjective analytic Toeplitz operator. Hence the results of Section \ref{S:SurToep} apply. In particular, $\sH'=\im T_G^*$ is a model space and the associated inner function $\tht$ is given by \eqref{thtform}. As before, we write $T'$ for the model operator $T'=P_{\sH'} S_p |_{\sH'}$.

Next we recall some results from \cite{FtHK-IEOT}. Set $\la=T_G^*(T_GT_G^*)^{-1}T_K$, viewed as an operator mapping $\ell^2_+(\BC^q)$ into $\sH'$. According to Lemma 2.3 in \cite{FtHK-IEOT}, the operator $\la$ is a strict contraction which satisfies
\[
T'\la=\la S_q.
\]
These two facts make it possible to apply commutant lifting theory. Following the argumentation in the last paragraph of Section 2 from \cite{FtHK-IEOT}, the contractive liftings of $\la$ that  intertwine $S_p$ and $S_q$ are precisely the Toeplitz operators defined by the solutions $X$ to the Leech problem associated with $G$ and $K$. Hence, the solutions are described in the appendix by Theorem \ref{cltstrict} as well as by Theorem \ref{cltstrict2}, specified to the special choice of $\la$ made here. {Note that this require $S_p$ to be a minimal isometric lifting of $T'$. Therefore (cf., Lemma \ref{lemmin1}) in what follows we shall assume that $\cap_{z\in\BD}\kr G(z)=\{0\}$.}

The following theorem is based on Theorem \ref{cltstrict2} specified for the case when the strict contraction $\la$ is given by $\la=T_G^*(T_GT_G^*)^{-1} {T_K}$.  Its prove require a number of non-trivial operator manipulations.

\begin{thm}\label{Allsol-LF}
Let $G\in H^\iy_{m\ts p}$ and $K\in H^\iy_{m\ts q}$ be such that $T_GT_G^*-T_KT_K^*$ is strictly positive, and assume that there is  no non-zero $x\in \BC^p$ such that $G(z)x$ is identically zero on the open unit disc $\BD$. Then the set of all solutions to the Leech problem \eqref{Leech1} associated with $G$ and $K$ is given by the range of the linear fractional map
\begin{equation}\label{LFTSol3}
X(z)=\big(\Up_{12}(z)+\Up_{11}(z)Y(z)\big) \big(\Up_{22}(z)+\Up_{21}(z)Y(z)\big)^{-1}, \quad |z|<1.
\end{equation}
Here $Y$ is an arbitrary function in $H^\iy_{k\ts q}$ with $\|Y\|_\infty\leq 1$, and
\begin{align}
\Up_{11}(z)&=\tht_0\de_1^{-1}-zE_p^*T_G^*(I-zS_m^*)^{-1}(T_GT_G^*-T_KT_K^*)^{-1}N\de_1^{-1}, \label{up11}\\
\Up_{21}(z)&=-zE_q^*T_K^*(I-zS_m^*)^{-1}(T_GT_G^*-T_KT_K^*)^{-1} N\de_1^{-1},\label{up21} \\
\Up_{12}(z)&= E_p^* T_G^*  (T_GT_G^* - T_KT_K^*)^{-1}T_KE_q \de_0^{-1} +\nn\\
&\hspace{.8cm}+  zE_p^* T_G^*(I - z S_m^*)^{-1}S_m^*(T_GT_G^* - T_KT_K^*)^{-1}T_KE_q \de_0^{-1},\label{up12} \\
\Up_{22}(z)&=\de_0 +zE_q^*T_K^*(I-zS_m^*)^{-1}S_m^*\ts \nn\\
&\hspace{3cm}\ts  (T_GT_G^*-T_KT_K^*)^{-1}T_KE_q\Delta_0^{-1}. \label{up22}
\end{align}
Here $\tht_0$   is a  one-to-one $p\ts k$ matrix  uniquely determined, up to multiplication with a constant unitary $k\ts k$ matrix from the right, by the identity \eqref{tht0}, and  $N=S_m^*T_GE_p\tht_0$, as in  Lemma \ref{L:thtids}. Furthermore, $\de_0$ and $\de_1$ are the positive definite matrices defined by

\begin{align}
\Delta_0^2 &=I_q+E_q^*T_K^*(T_GT_G^*-T_KT_K^*)^{-1}T_KE_q,\label{3del0}\\
\Delta_1^2 &=I_k+N^*\Big((T_GT_G^*-T_KT_K^*)^{-1}- (T_GT_G^*)^{-1}\Big)^{-1}N. \label{3del1}
\end{align}
\end{thm}

Before we proof the above theorem we recall two useful identities from \cite[Lemma 3.2]{FtHK-IEOT}:
\begin{align}
(I - \la^* \la)^{-1} &= I + T_K^* (T_GT_G^* - T_KT_K^*)^{-1}T_K, \label{FtHK7-1}\\
\la(I-\la^*\la)^{-1} &=  T_G^*(T_GT_G^* - T_KT_K^*)^{-1}T_K.  \label{FtHK7-2}
\end{align}

\begin{proof}[\bf Proof]
We split the proof into three parts. In the first part we  derive the identities  \eqref{up12} and \eqref{up22} using  formulas (3.7) and (3.8)  in \cite[Section 3]{FtHK-IEOT}. The  final two parts  contain the proofs of the formulas for  $\Up_{11}$ and $\Up_{21}$.

\smallskip
\paragraph{Part 1.}
From Theorem \ref{cltstrict2} we know that
\[
\Up_{12}(z)=U(z)\de_0^{-1} \ands  \Up_{22}(z)=U(z)\de_0^{-1},
\]
where $U$ and $V$ are given by \eqref{defU} and \eqref{defV}, respectively. From  formulas (3.7) and (3.8)  in \cite[Section 3]{FtHK-IEOT} we know that for our choice  of $\la$ the formulas \eqref{defU} and \eqref{defV} lead to the following identities:
\begin{align}
U(z)&= E_p^* T_G^*  (T_GT_G^* - T_KT_K^*)^{-1}T_KE_q  +\nn \\
&\hspace{.8cm}+  zE_p^* T_G^*(I - z S_m^*)^{-1}S_m^*(T_GT_G^* - T_KT_K^*)^{-1}T_KE_q, \label{3U}\\
V(z)&=I_q +E_q^* T_K^*  (T_GT_G^* - T_KT_K^*)^{-1}T_KE_q+\nn\\
&\hspace{.8cm}+zE_q^*T_K^*(I-zS_m^*)^{-1}S_m^*(T_GT_G^*-T_KT_K^*)^{-1}T_KE_q. \label{3V}
\end{align}
Furthermore, according  \eqref{FtHK7-1},   for our choice of $\la$ the matrix $\de_0^2$  is given by
\begin{align*}
\de_0^2&= E_q\Big(I + T_K^* (T_GT_G^* - T_KT_K^*)^{-1}T_K\Big)E_q\\
&=I_q+ E_q^*T_K^* (T_GT_G^* - T_KT_K^*)^{-1}T_KE_q.
\end{align*}
Hence $\de_0$ is the positive definite matrix  determined by \eqref{3del0}. Also note that $V(0)=\de_0^2$. But then  multiplying  \eqref{3U}  and  \eqref{3V} from the right by  $\de_0^{-1}$  we see that  $\Up_{12}$ and $\Up_{22}$ are given by \eqref{up12} and \eqref{up22}, respectively.

\smallskip
\paragraph{Part 2.}
In this part we derive the formula for $\Up_{21}$.  Recall from Theorem \ref{cltstrict2} that  $\Up_{21}(z)=zE_q^*(I-zS_q^*)^{-1}B_\nabla \de_1^{-1}$.  Using the adjoint of  \eqref{FtHK7-2} and  the operator $N$ introduced in Lemma \ref{L:thtids} we see that for our choice of $\la$, we have
\begin{align}
B_\nabla &= (I-\la^*\la)^{-1}\la^*S_p^*T_\tht E_k \nn\\
&=T_K^*(T_GT_G^*-T_KT_K^*)^{-1}T_G S_p^*T_\tht E_k \nn\\
&=-T_K^*(T_GT_G^*-T_KT_K^*)^{-1}N.\label{3Bnabla3}
\end{align}
Since $S_q^* T_K^*= T_K^* S_m^*$, it follows that
\begin{align*}
\Up_{21}(z)&=-zE_q^*(I-zS_q^*)^{-1}T_K^*(T_GT_G^*-T_KT_K^*)^{-1}N \de_1^{-1}\\
&=-zE_q^*T_K^*(I-zS_m^*)^{-1}(T_GT_G^*-T_KT_K^*)^{-1}N \de_1^{-1}.
\end{align*}
This proves \eqref{up21}.  It remains to show that $\de_1$ is determined by  \eqref{3del1}.

Using the definition of $\de_1^2$ in  \eqref{Deltas2}, our choice of $\la$   and the operator $N$ introduced in Lemma \ref{L:thtids}  we obtain
\begin{align}
\de_1^2&= I_k+E_k^* T_\tht^* S_p \la (I-\la^* \la)^{-1}\la^* S_p^* T_\tht E_k \nn\\
&= I_k+E_k^* T_\tht^* S_p T_G^*(T_GT_G^*)^{-1} T_K(I-\la^*\la)^{-1} T_K^*(T_GT_G^*)^{-1}\ts \nn\\
&\hspace{7cm}\ts  T_GS_p^*T_\tht E_k\nn\\
&= I_k+N^*(T_GT_G^*)^{-1}T_K(I-\la^*\la)^{-1} T_K^*(T_GT_G^*)^{-1}N. \label{tode1}
\end{align}
To complete the proof of  \eqref{3del1} it remains to show that
\begin{align}
&(T_GT_G^*)^{-1}T_K(I-\la^*\la)^{-1} T_K^*(T_GT_G^*)^{-1}=\nn\\
&\hspace{4cm}=(T_GT_G^*-T_KT_K^*)^{-1}- (T_GT_G^*)^{-1}. \label{3basicid}
\end{align}
This will be done in a few steps. We first show that for our choice of $\la$ we have
\begin{equation}
\label{3auxid}
T_K(I-\la^*\la)^{-1} T_K^*=T_GT_G^*(T_GT_G^*-T_KT_K^*)^{-1}T_KT_K^*.
\end{equation}
To see this note that
\[
T_K\Big(I-T_K^* (T_GT_G^*)^{-1}T_K\Big)=\Big(I-T_K T_K^*(T_GT_G^*)^{-1}\Big)T_K,
\]
and hence
\begin{align*}
T_K\Big(I-T_K^* (T_GT_G^*)^{-1}T_K\Big)^{-1}&=\Big(I-T_K T_K^*(T_GT_G^*)^{-1}\Big)^{-1}T_K\\
&=T_GT_G^*\Big(T_GT_G^*-T_K T_K^*\Big)^{-1}T_K.
\end{align*}
Thus, again using our choice of $\la$, we see that
\begin{align*}
T_K(I-\la^*\la)^{-1} T_K^*&=T_K\Big(I-T_K^*(T_GT_G^*)^{-1}T_G T_G^*(T_GT_G^*)^{-1} T_K\Big)^{-1}T_K^*\\
&=T_K\Big(I-T_K^*(T_GT_G^*)^{-1}T_K\Big)^{-1}T_K^*\\
&=T_GT_G^*\Big(T_GT_G^*-T_K T_K^*\Big)^{-1}T_KT_K^*,
\end{align*}
which proves \eqref{3auxid}. But then
\begin{align*}
&(T_GT_G^*)^{-1}T_K(I-\la^*\la)^{-1} T_K^*(T_GT_G^*)^{-1}=\\
&\hspace{1cm}=(T_GT_G^*)^{-1}\Big[T_GT_G^*\Big(T_GT_G^*-T_K T_K^*\Big)^{-1}T_KT_K^*\Big]     (T_GT_G^*)^{-1}\\
&\hspace{1cm}= \Big(T_GT_G^*-T_K T_K^*\Big)^{-1}T_KT_K^*      (T_GT_G^*)^{-1}\\
&\hspace{1cm}=\Big(T_GT_G^*-T_K T_K^*\Big)^{-1}\Big(T_GT_G^*-(T_GT_G^*-T_K T_K^*)\Big) (T_GT_G^*)^{-1}\\
&\hspace{1cm}= (T_GT_G^*-T_K T_K^* )^{-1}- (T_GT_G^*)^{-1}.
\end{align*}
This proves \eqref{3basicid}. Using the identity \eqref{3basicid} in \eqref{tode1} yields   \eqref{3del1}.

\smallskip
\paragraph{Part 3.}
In this part we derive the formula for $\Up_{11}$.  Using our choice of  $\la$,  the formula for $\tht$ given by \eqref{thtform}, and the first identity in \eqref{LFUps} we see that
\[
\Up_{11}(z)-{\tht_0}\de_1^{-1}= A(z) +B(z)+ C(z),
\]
where
\begin{align*}
A(z)&=-zE_p^*T_G^*(I-zS_m^*)^{-1}(T_GT_G^*)^{-1}N\de_1^{-1},\\
B(z)&=zE_p^*(I-zS_p^*)^{-1}T_G^*(T_GT_G^*)^{-1}T_K E_qE_q^*(I-zS_q^*)^{-1}B_\nabla\de_1^{-1},\\
C(z)&=- z\tht(z)E_k^*T_\tht^*S_pT_G^*(T_GT_G^*)^{-1}T_K(I-zS_q^*)^{-1}S_q^* B_\nabla\de_1^{-1}.
\end{align*}
Here $B_\nabla=-T_K^*(T_GT_G^*-T_KT_K^*)^{-1}N$,  with $N=-T_GS_p^*T_\tht E_k$ as in Lemma~\ref{L:thtids}, and  $\de_1$ is the positive definite matrix determined by  \eqref{3del1}.

First we deal with $C(z)$. Using the formula for $N$ and the identity  \eqref{thtid} we see that
\begin{align*}
C(z)&=z\tht(z)N^*(T_GT_G^*)^{-1}T_K(I-zS_q^*)^{-1}S_q^*B_\nabla\de_1^{-1}\\
&=zE_p^*(I-zS_p^*)^{-1}T_G^*(T_GT_G^*)^{-1}(I-zS_m^*)S_m\ts\\
&\hspace{5cm}\ts T_K(I-zS_q^*)^{-1}S_q^*B_\nabla\de_1^{-1}.
\end{align*}
Note that $(I-zS_m^*)S_m=S_m-zI$, and hence $C(z)=C_1(z)+C_2(z)$, where
\begin{align*}
C_1(z)&= zE_p^*(I-zS_p^*)^{-1}T_G^*(T_GT_G^*)^{-1}S_mT_K (I-zS_q^*)^{-1}S_q^*B_\nabla\de_1^{-1},\\
C_2(z)&=-z^2 E_p^*(I-zS_p^*)^{-1}T_G^*(T_GT_G^*)^{-1}T_K (I-zS_q^*)^{-1}S_q^*B_\nabla\de_1^{-1}.
\end{align*}
Next we use the intertwining relation  $T_KS_q=S_mT_K$ and the identity
\[
S_q(I-zS_q^*)^{-1}S_q^*= S_qS_q^*(I-zS_q^*)^{-1}.
\]
This yields
\[
C_1(z)=zE_p^*(I-zS_p^*)^{-1}T_G^*(T_GT_G^*)^{-1}T_KS_qS_q^*(I-zS_q^*)^{-1}B_\nabla\de_1^{-1},
\]
and hence, using $E_qE_q^*+S_qS_q^*=I$, we obtain
\[
B(z)+C_1(z)=zE_p^*(I-zS_p^*)^{-1}T_G^*(T_GT_G^*)^{-1}T_K(I-zS_q^*)^{-1}B_\nabla\de_1^{-1}.
\]
Next observe that
\begin{align*}
C_2(z)&=z E_p^*(I-zS_p^*)^{-1}T_G^*(T_GT_G^*)^{-1}T_K (I-zS_q^*)^{-1}(-zS_q^*)B_\nabla\de_1^{-1}\\
&=zE_p^*(I-zS_p^*)^{-1}T_G^*(T_GT_G^*)^{-1}T_K (I-zS_q^*)^{-1}\ts\\
&\hspace{5cm}\ts \Big((I-zS_q^*)-I\Big)B_\nabla\de_1^{-1}\\
&=C_{21}(z)+C_{22}(z),
\end{align*}
where
\begin{align*}
C_{21}(z)&=zE_p^*(I-zS_p^*)^{-1}T_G^*(T_GT_G^*)^{-1}T_K B_\nabla\de_1^{-1},\\
C_{22}(z)&=-zE_p^*(I-zS_p^*)^{-1}T_G^*(T_GT_G^*)^{-1}T_K (I-zS_q^*)^{-1}B_\nabla\de_1^{-1}.
\end{align*}
We conclude  that $B(z)+C_1(z)+C_{22}(z)=0$, and hence
\[
\Up_{11}(z)-{\tht_0}\de_1^{-1}= A(z) +C_{21}(z).
\]

Next, using the intertwining relation $S_mT_G=T_G S_p$ and the formula for $B_\nabla$ given by   \eqref{3Bnabla3} we see that
\[
C_{21}(z)=-zE_p^*T_G^*(I-zS_m^*)^{-1}(T_GT_G^*)^{-1}T_KT_K^*(T_GT_G^*-T_KT_K^*)^{-1}N\de_1^{-1}.
\]
But then
\[
A(z)+C_{21}(z)=-zE_p^*T_G^*(I-zS_m^*)^{-1}M N\de_1^{-1},
\]
where
\begin{align*}
M&=(T_GT_G^*)^{-1}+(T_GT_G^*)^{-1}T_KT_K^*(T_GT_G^*-T_KT_K^*)^{-1}\\
&=(T_GT_G^*)^{-1}\Big((T_GT_G^*-T_KT_K^*)+T_KT_K^*\Big)(T_GT_G^*-T_KT_K^*)^{-1}\\
&=(T_GT_G^*-T_KT_K^*)^{-1}.
\end{align*}
Thus $\Up_{11}(z)={\tht_0}\de_1^{-1} +A(z)+C_{21}(z)$ is equal to the right hand sight of  the \eqref{up11}, and hence the identity \eqref{up11} is proved.
\end{proof}


\begin{remark} We conclude this section with a remark about the coefficients $\Up_{ij}$, $1\leq i, j \leq 2$, in the linear fractional map \eqref{LFTsols}. Since each $X$ given by \eqref{LFTSol3} is a solution to the Leech problem \eqref{Leech1}  associated with $G$ and $K$  we see that
\[
G(z)\Big(\Up_{12}(z)+\Up_{11}(z)Y(z)\Big)=K(z)\Big(\Up_{22}(z)+\Up_{21}(z)Y(z)\Big)
\]
for each $Y$  in $H^\iy_{k\ts q}$ with $\|Y\|_\infty\leq 1$. The previous identity can be rewritten as
\[
G(z)\Up_{12}(z)-K(z)\Up_{22}(z)= -\Big(G(z)\Up_{11}(z)-K(z)\Up_{21}(z)\Big)Y(z).
\]
Using the freedom in the choice of $Y$, we see that the following proposition holds.
\end{remark}

\begin{prop}\label{propGKUps} The functions $\Up_{ij}$, $1\leq i, j \leq 2$, given by \eqref{up11} --  \eqref{up22} satisfy the following identities:
\begin{equation}
\label{idents}
G(z)\Upsilon_{1j}(z)-K(z)\Upsilon_{2j}(z)= 0, \quad z\in \BD\quad (j=1,2).
\end{equation}
\end{prop}

We use the remaining part of this section to give  a direct proof of the two identities in \eqref{idents}. We begin with two lemmas.

\begin{lem}\label{lemids1} The following identities hold:
\begin{align}
&T_GE_pE_p^*T_G^*=T_G T_G^*- S_mT_G T_G^*S_m^*, \label{idTG}\\
&T_KE_qE_q^*T_K^*=T_K T_K^*- S_mT_K T_K^*S_m^*,  \label{idTK}\\
&E_m^*(I-zS_m^*)^{-1}S_m=zE_m^*(I-zS_m^*)^{-1}\quad (z\in \BD). \label{idS1}
\end{align}
Furthermore, for any $z\in \BD$ and  any bounded linear operator $X$ on $\ell_+^2(\BC^m)$ we have
\begin{equation}
E_m^*(I-zS_m^*)^{-1}\Big(X-S_mXS_m^*\Big) (I-zS_m^*)^{-1}=E_m^*(I-zS_m^*)^{-1}X.\label{idS2}
\end{equation}
\end{lem}

\begin{proof}[\bf Proof.]
Note that $E_pE_p^*=I-S_pS_p^*$. Since $T_G$ is a block lower triangular operator $T_GS_p=S_m T_G$, and  $S_p^*T_G^*=T_G^*S_m^*$ by duality. From these remarks \eqref{idTG} is clear. The identity \eqref{idTK} is proved in the same way.

The identity \eqref{idS1} follows from $S_m^*S_m=I$ and $E_m^* S_m=0$. Indeed, using the latter two identities, we see that
\begin{align*}
E_m^*(I-zS_m^*)^{-1}S_m &=E_m^*\big(I+z(I-zS_m^*)^{-1}S_m^*\big) S_m\\
&=E_m^* \Big(S_m+z(I-zS_m^*)^{-1}S_m^*S_m\Big)=zE_m^*(I-zS_m^*)^{-1}.
\end{align*}
Finally, to obtain \eqref{idS2}  we use \eqref{idS1}. Indeed
\begin{align*}
&E_m^*(I-zS_m^*)^{-1}\Big(X-S_mXS_m^*\Big) (I-zS_m^*)^{-1}=\\
&\hspace{.5cm}= E_m^*(I-zS_m^*)^{-1}X (I-zS_m^*)^{-1}+\\
&\hspace{3cm}-E_m^*(I-zS_m^*)^{-1}S_mXS_m^*(I-zS_m^*)^{-1}\\
&\hspace{.5cm}= E_m^*(I-zS_m^*)^{-1}X (I-zS_m^*)^{-1}+\\
&\hspace{3cm}-zE_m^*(I-zS_m^*)^{-1}XS_m^*(I-zS_m^*)^{-1}\\
&\hspace{.5cm}= E_m^*(I-zS_m^*)^{-1}X (I-zS_m^*)^{-1}+\\
&\hspace{3cm}+E_m^*(I-zS_m^*)^{-1}X(I-zS_m^*-I)(I-zS_m^*)^{-1}\\
&\hspace{.5cm}= E_m^*(I-zS_m^*)^{-1}X,
\end{align*}
which completes the proof.
\end{proof}

\begin{lem}\label{lemids2} Put $\de= T_GT_G^*- T_KT_K^* $, and let
\begin{equation}
\label{defAB}
A(z)=E_p^*T_G^*(I-zS_m^*)^{-1} , \quad  B(z)=E_q^*T_K^*(I-zS_m^*)^{-1}.
\end{equation}
Then
\begin{equation}
\label{idGK1}
G(z)A(z)-K(z)B(z)=E_m^*(I-zS_m^*)^{-1}\de\quad (z\in \BD).
\end{equation}
\end{lem}

\begin{proof}[\bf Proof.]
First note that  that $G$ and $K$ admit the following infinite dimensional realizations:
\begin{align}
G(z)&=E_m^*(I-zS_m^*)^{-1}T_GE_p \quad (z\in \BD), \label{realGinf}\\
K(z)&=E_m^*(I-zS_m^*)^{-1}T_KE_q \quad(z\in \BD)  \label{realKinf}.
\end{align}
Using \eqref{realGinf}, the definition of $A(z)$ in \eqref{defAB}, and the identity  \eqref{idTG}, we see that
\begin{align*}
G(z)A(z)&=E_m^*(I-zS_m^*)^{-1}T_GE_pE_p^*T_G^*(I-zS_m^*)^{-1} \\
&=E_m^*(I-zS_m^*)^{-1}(T_G T_G^*- S_mT_G T_G^*S_m^*) (I-zS_m^*)^{-1}.
\end{align*}
Similarly, using \eqref{realKinf}, the definition of $B(z)$ in \eqref{defAB}, and the identity  \eqref{idTK}, we get
\begin{align*}
K(z)B(z)&=E_m^*(I-zS_m^*)^{-1}T_KE_qE_q^*T_K^*(I-zS_m^*)^{-1}\\
&=E_m^*(I-zS_m^*)^{-1}(T_K T_K^*- S_mT_K T_K^*S_m^*) (I-zS_m^*)^{-1}.
\end{align*}
Applying \eqref{idS2}, first with $X=T_G T_G^*$  and next with $X=T_K T_K^*$, we conclude that
\begin{align}
G(z)A(z)&=E_m^*(I-zS_m^*)^{-1}T_G T_G^*  \quad (z\in \BD),\label{eqGA}\\
K(z)B(z)&=E_m^*(I-zS_m^*)^{-1}T_KT_K^*\quad (z\in \BD). \label{eqKB}
\end{align}
Taking the difference yields \eqref{idGK1}.
\end{proof}

\begin{proof}[\bf Proof of  Proposition \ref{propGKUps}.]
We split the proof into two parts. As  in the preceding lemma, $\de= T_GT_G^*- T_KT_K^* $. Furthermore, throughout $z\in \BD$.

\smallskip
\paragraph{Part 1.} We prove the identity \eqref{idents} for $j=1$. Using the formula for $\tht$ in \eqref{thtform} we see that $\Up_{11}$ can be rewritten in the following equivalent form:
\begin{align*}
\Up_{11}(z)&=\tht(z)\de_1^{-1}+zE_p^*T_G^*(I-zS_m^*)^{-1}(T_GT_G^*)^{-1}N\de_1^{-1}+\\
&\hspace{1cm} -zE_p^*T_G^*(I-zS_m^*)^{-1}(T_GT_G^*-T_KT_K^*)^{-1}N\de_1^{-1}
\end{align*}
The fact that $\im T_\tht=\kr T_G$ implies that $G(z)\tht(z)=0$, and hence, using the definition of $A(z)$ in \eqref{defAB}, we see that
\begin{align*}
G(z)\Up_{11}(z)&=zG(z)A(z)\Big((T_GT_G^*)^{-1}-(T_GT_G^*-T_KT_K^*)^{-1}\Big)N\de_1^{-1}.
\end{align*}
Next, using the definition of $B(z)$ in \eqref{defAB},  we obtain
\[
K(z)\Up_{21}(z)=-zK(z)B(z)(T_GT_G^*-T_KT_K^*)^{-1} N\de_1^{-1}.
\]
Taking the difference, applying \eqref{idGK1} and using \eqref{realKinf},  we get
\begin{align}
&G(z)\Upsilon_{12}(z)-K(z)\Upsilon_{22}(z)=\nonumber\\
&\hspace{.5cm}=zG(z)A(z)  (T_GT_G^*)^{-1}N\de_1^{-1}-z\big (G(z)A(z)-K(z)B(z)\big)\de^{-1}N\de_1^{-1}\nonumber\\
&\hspace{.5cm}=zG(z)A(z)  (T_GT_G^*)^{-1}N\de_1^{-1}-zE_m^*(I-zS_m^*)^{-1} N\de_1^{-1}.\label{GKUps2}
\end{align}
According to  \eqref{eqGA} we have $G(z)A(z)  (T_GT_G^*)^{-1}=E_m^*(I-zS_m^*)^{-1}$. Using the latter identity  in \eqref{GKUps2}, we see that \eqref{idents} holds for $j=1$.

\smallskip
\paragraph{Part 2.}  We prove the identity \eqref{idents} for $j=2$. Note that \eqref{up12} and \eqref{up22} can be rewritten in the following equivalent form;
\begin{align*}
\Upsilon_{12}(z)&=E_p^* T_G^*(I - z S_m^*)^{-1} (T_GT_G^* - T_KT_K^*)^{-1}T_KE_q \de_0^{-1},\\
\Upsilon_{12}(z)&=\de_0^{-1}+ E_q^*T_K^*(I-zS_m^*)^{-1} (T_GT_G^*-T_KT_K^*)^{-1}T_KE_q\Delta_0^{-1}.
\end{align*}
Using \eqref{defAB} and  the above formulas  for $\Up_{12}$ and $\Up_{22}$,  we see that
\begin{align*}
G(z)\Upsilon_{12}(z)&=G(z)A(z)\de^{-1} (T_KE_q\de_0^{-1}),\\
K(z)\Upsilon_{22}(z)&=K(z)\de_0^{-1}+K(z)B(z)\de^{-1} (T_KE_q\de_0^{-1}).
\end{align*}
Taking the difference, applying \eqref{idGK1} and using \eqref{realKinf},  we obtain
\begin{align*}
&G(z)\Upsilon_{12}(z)-K(z)\Upsilon_{22}(z)=\\
&\hspace{.5cm}=\big (G(z)A(z)-K(z)B(z)\big)\de^{-1} (T_KE_q\de_0^{-1})-K(z)\de_0^{-1}\\
&\hspace{.5cm}=E_m^*(I-zS_m^*)^{-1}\de \de^{-1}(T_KE_q\de_0^{-1})-K(z)\de_0^{-1}\\
&\hspace{.5cm}=E_m^*(I-zS_m^*)^{-1} (T_KE_q\de_0^{-1})-E_m^*(I-zS_m^*)^{-1}T_KE_q\de_0^{-1} =0.
\end{align*}
This completes the proof.
\end{proof}



\section{State space computations}\label{S:SSF}

In this section we prove Theorem  \ref{T:main1}. To this end, we first recall some formulas derived in \cite{FtHK-IEOT}. Let $G\in \fR H^\iy_{m\ts p}$ and $K\in \fR H^\iy_{m\ts q}$ be given by the realization of $\begin{bmatrix} G&K\end{bmatrix}$ in \eqref{GKreal}. Assume $T_GT_G^*-T_KT_K^*$ is strictly positive. Then there exist stabilizing solutions $Q$ and $Q_0$ to the Riccati equations \eqref{are} and \eqref{are2}, respectively. Let $P_1$ and $P_2$ be the controllability gramians that solve the Stein equations \eqref{Stein} for $j=1,2$. Define $\de$ and  $A_0$ by \eqref{A0Delta}, the matrices $C_j$, for $j=0,1,2$, $B_0$,  and $\de_{j}$, for $j=0,1$, as in Theorem~\ref{T:main1}. Furthermore, as in Theorem~\ref{T:main1}, the matrix $\om$ is given by
\[
\om=(P_1-P_2)(Q^{-1}+P_2-P_1)^{-1}Q^{-1}.
\]

Now, write $W_{obs}$ and $W_0$ for the observability operators defined by the pairs $\{C,A\}$ and $\{C_0,A_0\}$, respectively, that is,
\[
W_{obs}=\mat{c}{C\\CA\\CA^2\\\vdots},\quad
W_{0}=\mat{c}{C_0\\C_0A_0\\C_0A_0^2\\\vdots}.
\]
The following identities are covered by  \cite[Eq.(5.9)]{FtHK-IEOT} and  \cite[Eq.(5.5)]{FtHK-IEOT} :
\begin{align}\label{FtHK7SSids1}
E_p^* T_G^*W_0=C_1,\quad E_q^* T_K^*W_0=C_2,\quad Q=W_{obs}^*W_0.
\end{align}
Moreover, according to  the comment directly after  \cite[Eq.(5.7)]{FtHK-IEOT} we have
\begin{align}\label{FtHK7SSids2}
S_m^*(T_GT_G^*-T_KT_K^*)^{-1}T_KE_q=W_0B_0.
\end{align}
Finally, let $R$ be the function given by \eqref{defR} and $T_R$ the Toeplitz operator associated with $R$. Recall that $T_GT_G^*-T_KT_K^*$ strictly positive implies $T_R$ is strictly positive. Then Theorem 1.1 in \cite{FtHK-IEOT} yields
\[
(T_GT_G^*-T_KT_K^*)^{-1}=T_R^{-1}+T_R^{-1}W_{obs}\om W_{obs}^* T_R^{-1}.
\]
Along with
\begin{equation}
W_0=T_R^{-1}W_{obs}, \label{FtHK7SSids3}
\end{equation}
which was proved in \cite[Lemma 5.1]{FtHK-IEOT}, this shows that
\begin{align}\label{FtHK7SSids4}
(T_GT_G^*-T_KT_K^*)^{-1}\!\!=T_R^{-1}+W_0\om W_0^*.
\end{align}
Note that \eqref{FtHK7SSids3} also shows that $Q=W_{obs}^*T_R^{-1}W_{obs}$,  by the third identity in~\eqref{FtHK7SSids1}.

Using the formulas in \eqref{FtHK7SSids1} and \eqref{FtHK7SSids2} the state space representations of $\Up_{12}$ and $\Up_{22}$ in Theorem \ref{T:main1}  follow immediately. In fact,  as we have seen before (Part 1 of the proof of Theorem \ref{Allsol-LF}),  $\Up_{12}$ and $\Up_{22}$ are related to $U$ and $V$ in \cite{FtHK-IEOT} through $\Up_{12}\equiv U\de_0^{-1}$ and $\Up_{22}\equiv V\de_0^{-1}$, and the formulas for $\Up_{12}$ and $\Up_{22}$ in Theorem \ref{T:main1} above follow directly from the formulas for $U$ and $V$ derived in \cite{FtHK-IEOT}; see    \cite[Eq. (5.14)]{FtHK-IEOT} and  \cite[Eq. (5.13)]{FtHK-IEOT}, respectively.

In order to show that $\Up_{11}$ and $\Up_{21}$, the two remaining functions  in Theorem \ref{Allsol-LF}, admit the desired finite dimensional state space realizations  requires a bit more work.

\begin{proof}[\bf Proof of Theorem  \ref{T:main1}.]
In order to complete the proof  of Theorems \ref{T:main1}  it suffices to show that $\Up_{11}$ in \eqref{up11} and $\Up_{21}$ in  \eqref{up21} admit finite dimensional state space representations as in \eqref{UpsFin} and that the positive definite matrices $\de_0$ and $\de_1$ defined by \eqref{3del0} and \eqref{3del1} are also given by \eqref{DeltasFin}. Note that in Theorem \ref{T:main1} as well as in Theorem \ref{Allsol-LF} we assume  that there is  no non-zero $x\in \BC^p$ such that $G(z)x$  is identically zero on the open unit disc $\BD$.

In order to compute the remaining state space formulas, we prove the following identity:
\begin{equation}\label{newform}
(T_GT_G^*-T_KT_K^*)^{-1}N=W_0 Q^{-1}(Q^{-1}+P_2-P_1)^{-1}B_1 {\tht_0}.
\end{equation}
First observe that
\begin{align}\label{Nss}
N=S_m^* T_G E_p{\tht_0}=W_{obs}B_1 {\tht_0}.
\end{align}
Now, combining \eqref{FtHK7SSids4} and \eqref{FtHK7SSids3} along with the third identity in \eqref{FtHK7SSids1} we obtain that
\begin{align*}
&(T_GT_G^*-T_KT_K^*)^{-1}W_{obs}=\\
&\qquad\qquad=T_R^{-1}W_{obs}+W_0(P_1-P_2)(Q^{-1}+P_2-P_1)^{-1}Q^{-1}W_0^*W_{obs}\\
&\qquad\qquad=W_0+W_0(P_1-P_2)(Q^{-1}+P_2-P_1)^{-1}\\
&\qquad\qquad=W_0(I+(P_1-P_2)(Q^{-1}+P_2-P_1)^{-1})\\
&\qquad\qquad=W_0(Q^{-1}+P_2-P_1+P_1-P_2)(Q^{-1}+P_2-P_1)^{-1}\\
&\qquad\qquad=W_0Q^{-1}(Q^{-1}+P_2-P_1)^{-1}.
\end{align*}
Together with \eqref{Nss} this gives \eqref{newform}.

Using \eqref{newform} along with $S_m^* W_0=W_0 A_0$ we obtain
\begin{align*}
\Up_{11} (z)
&=\tht_0\Delta_1^{-1}-zE_p^*T_G^*(I-z S_m^*)^{-1}(T_GT_G^*-T_KT_K^*)^{-1}N\Delta_1^{-1}\\
&=\tht_0\Delta_1^{-1}-zE_p^*T_G^*(I-z S_m^*)^{-1}W_0 Q^{-1}(Q^{-1}+P_2-P_1)^{-1}B_1 {\tht_0}\Delta_1^{-1}\\
&=\tht_0\Delta_1^{-1}-zE_p^*T_G^*W_0(I-z A_0)^{-1} Q^{-1}(Q^{-1}+P_2-P_1)^{-1}B_1 {\tht_0}\Delta_1^{-1}\\
&=\tht_0\Delta_1^{-1}-zC_1(I-z A_0)^{-1} Q^{-1}(Q^{-1}+P_2-P_1)^{-1}B_1 {\tht_0}\Delta_1^{-1}.
\end{align*}
To obtain the  last equality we used  the first equality in \eqref{FtHK7SSids1}.  Similarly
\begin{align*}
\Up_{21} (z)
&=-zE_q^*T_K^*(I-z S_m^*)^{-1}(T_GT_G^*-T_KT_K^*)^{-1}N\Delta_1^{-1}\\
&=-zE_q^*T_K^*W_0(I-z A_0)^{-1} Q^{-1}(Q^{-1}+P_2-P_1)^{-1}B_1\Delta_1^{-1}\\
&=-zC_2(I-z A_0)^{-1} Q^{-1}(Q^{-1}+P_2-P_1)^{-1}B_1\Delta_1^{-1}.
\end{align*}
In the final step of the above computation we used the second equality in \eqref{FtHK7SSids1}.

The computations above show that $\Up_{11}$ and $\Up_{21}$ admit the state space representation given by in \eqref{UpsFin}. It remains to show that $\de_0$ and $\de_1$ are the positive definite matrices determined by \eqref{DeltasFin}. The matrix $\de_0$ in fact appears in \cite{FtHK-IEOT}, denoted by $D_V$ in \cite[Eq.(3.4)]{FtHK-IEOT}, and a formula in terms of the state space realization \eqref{GKreal} and related matrices is given in \cite[Eq.(1.16)]{FtHK-IEOT}. We derive here a different formula, given in \eqref{DeltasFin} above, which better exhibits the positive definite character.

Recall from \eqref{3del0} that
\[
\de_0^2=I_q+E_q^*T_K^*(T_GT_G^*-T_KT_K^*)^{-1}T_KE_q.
\]
Using \eqref{FtHK7SSids4} and the second identity in \eqref{FtHK7SSids1} we obtain that
\[
\de_0^2=I_q+C_2\Omega C_2^*+E_q^* T_K^* T_R^{-1}T_K E_q.
\]
Recall that on page 14 of \cite{FtHK-IEOT} it was shown that
\begin{align*}
T_R^{-1}
&=\mat{cc}{\de^{-1} &  -\de^{-1}\Gamma^* W_0^*\\ -W_0 \Gamma \de^{-1} & T_R^{-1}+W_0 \Gamma \Delta^{-1}\Gamma^* W_0^*}\\
&=\mat{c}{I_m\\-W_0\Gamma} \Delta^{-1}\mat{cc}{I_m & -\Gamma^* W_0^*}+\mat{cc}{0&0\\0&T_R^{-1}}.
\end{align*}
Recall that $W_0=T_R^{-1}W_{obs}$, see \eqref{FtHK7SSids3}. Since
\[
T_KE_q=\begin{bmatrix}D_2\\ W_{obs}B_2\end{bmatrix}\ands Q=W_{obs}^*W_0=W_{obs}^*T_R^{-1}W_{obs},
\]
we obtain that
\[
E_q^* T_K^* T_R^{-1}T_K E_q=(D_2-\Gamma^* Q B_2)^* \Delta^{-1} (D_2-\Gamma^* Q B_2) + B_2^* Q B_2.
\]
Therefore, we have
\[
\de_0^2=I_q+C_2\Omega C_2^*+(D_2-\Gamma^* Q B_2)^* \Delta^{-1} (D_2-\Gamma^* Q B_2) + B_2^* Q B_2,
\]
as claimed.

Recall (see \eqref{3del1}) that $\de_1$ is be given by
\begin{equation*}
\de_1^2=I_k+N^*(T_GT_G^*-T_KT_K^*)^{-1}N-N^*(T_GT_G^*)^{-1}N.
\end{equation*}
Using \eqref{newform} we obtain that
\[
N^*(T_GT_G^*-T_KT_K^*)^{-1}N=N^*W_0 Q^{-1}(Q^{-1}+P_2-P_1)^{-1}B_1 {\tht_0}.
\]
By \eqref{Nss} and the third identity in \eqref{FtHK7SSids1} we have $N^*W_0={\tht_0}^*B_1^*Q$. This yields
\begin{equation}\label{De1form1}
N^*(T_GT_G^*-T_KT_K^*)^{-1}N={\tht_0}^*B_1^*(Q^{-1}+P_2-P_1)^{-1}B_1 {\tht_0}.
\end{equation}
For the last summand in the formula of $\de_1^2$ we have to consider the Leech problem \eqref{Leech1} with $K\equiv0$. In that case $P_2=0$ and we write $Q_0$ for the solution to the associated Riccati equation \eqref{are2}. Since the operator $N=S_m^*T_GE_p{\tht_0}$ does not involve $K$, translating \eqref{De1form1} to the case $K\equiv0$ yields
\begin{equation}\label{De1form2}
N^*(T_GT_G^*)^{-1}N={\tht_0}^*B_1^*(Q_0^{-1}-P_1)^{-1}B_1 {\tht_0}.
\end{equation}
Inserting \eqref{De1form1} and \eqref{De1form2} into the formula for $\de_1^2$ derived above gives the formula for $\de_1^2$ in \eqref{DeltasFin}.
\end{proof}

\begin{remark}
Two important special cases of the Leech problem are the Toeplitz corona problem, which can be reduced to the case where $q=m$ and $K$ is identically equal to the identity matrix $I_m$ ($K\equiv I_m$), and the case where $K$ is identically equal to the zero matrix ($K\equiv 0$). On the level of the state space representation \eqref{GKreal} these correspond to the cases $B_2=0$ and $D_2=I_m$, and $B_2=0$ and $D_2=0$, respectively. Recall that the scalar corona problem was proved by Carlson \cite{Carl62} and the matrix case by Fuhrmann \cite{Fuhr68}; see \cite{Peller03} for a discussion of the problem. For the Toeplitz corona problem, Theorem \ref{T:main1} leads to a description of the solutions via a similar linear fractional transformation. We omit the precise formulas for the coefficients $\Upsilon_{ij}$, $i,j=1,2$, and only mention some of the matrices appearing in  Theorem \ref{T:main1} that simplify:
\begin{align*}
&P_2=0,\quad\Gamma=\Gamma_0,\quad C_2=C_0,\quad B_0=A_0\Omega-\Gamma\Delta^{-1},\\
&\hspace{.9cm}\Delta_0^2=I_q+C_0\Omega C_0^* + \Delta^{-1}, \quad  \Delta_1=I_{m-p}.
\end{align*}

The situation is different for the case $K\equiv 0$, i.e., $B_2=0$ and $D_2=0$. Then
\begin{align*}
P_2=0,\quad\Gamma=\Gamma_0,\quad C_2=0,\quad B_0=0,\quad
\Delta_0=I_q,\quad \Delta_1=I_{m-p}.
\end{align*}
From these formulas one immediately obtains that
\[
\Upsilon_{12}(z)=0,\quad
\Upsilon_{21}(z)=0,\quad
\Upsilon_{22}(z)=I_{q}\quad (z\in\BD).
\]
The formula for $\Upsilon_{11}$ reduces to
\[
\Theta_0-zC_1(I-zA_0)^{-1}Q_0^{-1}(Q_0^{-1}-P_1)B_1\Theta_0\quad (z\in\BD)
\]
where $Q_0$ is the stabilizing solution to the Riccati equation \eqref{are2} and
\[
A_0=A-\Gamma_0(R_{10}-\Gamma_0^*Q_0\Gamma_0)^{-1}(C-\Gamma_0^*Q_0A).
\]
On inspection of the formula for $\Upsilon_{11}$ given in Section \ref{S:InfDim}, we see that
\begin{align*}
\Up_{11}(z)&=\tht_0-zE_p^*T_G^*(I-zS_m^*)^{-1}(T_GT_G^*)^{-1}N=\Theta(z),
\end{align*}
where $\Theta$ is the inner function in $H^\iy_{p\ts (p-m)}$ such that $\kr T_G=\im T_\Theta$, see Lemma \ref{L:thtids}. Hence, as expected,  the solutions to the Leech problem \eqref{Leech1} with $K\equiv 0$ are given by $X=\Theta Y$ with $Y$ an arbitrary function in $H^\infty_{(p-m)\ts m}$ with $\|Y\|_\infty\leq 1$.
\end{remark}

\appendix
\section{Commutant lifting}\label{App}
\renewcommand{\theequation}{A.\arabic{equation}}

In this appendix we derive a version of the commutant lifting theorem, based on Theorem VI.6.1 in \cite{FFGK98}, which we need for the proof of our main results.

We begin with some notation. Throughout this appendix $\sH'$ is a subspace of $\ell_+^2(\BC^p)$, invariant under the backward shift $S_p^*$ on $\ell_+^2(\BC^p)$. The latter means there exists an inner function $\tht\in H^\iy_{p\ts k}$ for some positive integer $k\leq p$ such that $\sH'=\kr T_\tht^*$, that is,
\begin{equation}
\label{decom}
\ell_+^2(\BC^p) =\sH'\oplus T_\tht \ell^2_+(\BC^k).
\end{equation}
By $T'$  we denote the compression of the forward shift $S_p$ on $H^2_p$ to $\sH'$.  It follows that $S_p$ admits the following operator $2\ts 2$ block operator matrix representation for appropriate choices of $W$ and $Z$:
\begin{equation} \label{lifting2}
S_p=\begin{bmatrix}
   T'   &    0\\
     W &  Z
\end{bmatrix}\mbox{ on } \begin{bmatrix}
  \sH'\\
  \im T_\tht
\end{bmatrix}.
\end{equation}
Hence $S_ p$ is an isometric lifting of $T'$. The first theorem in this appendix is the following variation on Theorem VI.6.1 in \cite{FFGK98} for the isometric lifting $S_p$ of $T'$. We shall assume that $S_p$ is a minimal isometric lifting of $T'$, that is,
\[
\ell^2_+(\BC^p)=\bigvee_{\nu\geq 0}S_p^\nu \sH.
\]

\begin{thm}\label{cltstrict}
Assume $S_ p$ is a minimal isometric lifting of $T^\prime$, and let $\la$ be a strict contraction mapping $\ell_+^2(\BC^q)$ into $\sH^\prime\subset \ell_+^2(\BC^p)$ satisfying the intertwining relation $T^\prime \la= \la S_q$. Then   all functions $X$ in $H_{p\times q}^\infty$ satisfying
\begin{equation}\label{clt2a}
\la = P_{\sH^\prime} T_X
\quad \mbox{and}\quad \|X\|_\infty \leq 1
\end{equation}
are given by
\begin{equation}\label{RedSolCLT}
X(z)=\Phi_{22}(z)+\Phi_{21}(z)Y(z)\big(I-\Phi_{11}(z)Y(z)\big)^{-1}\Phi_{12}(z). \quad |z|<1.
\end{equation}
Here $Y$ is an arbitrary function in $H^\iy_{k\ts q}$ with $\|Y\|_\infty\leq 1$, and
\begin{align}
\notag \Phi_{11}(z)
&=-z\Delta_0^{-1}E_q^*(I-z\BM)^{-1}B_{\nabla}\Delta_1^{-1},\\
\label{PhisCLT1}
\Phi_{12}(z) &=\Delta_0^{-1}E_q^*(I-z\BM)^{-1}E_q,\\
\notag \Phi_{21}(z)
&=\Theta(z)\Delta_1- \Theta(z)E_k^* T_\Theta^* S_p
\Lambda(I-z\BM)^{-1}B_{\nabla}\Delta_1^{-1},\\
\notag \Phi_{22}(z)
&=E_p^*(I-zS_p^*)^{-1}\Lambda E_q+\Theta(z)E_k^* T_\Theta^*S_p \Lambda
\BM(I-z\BM)^{-1}E_q.
\end{align}
Here $\BM$ is the operator on $\ell_+^2(\BC^q)$, with spectral radius $r_{spec}(\BM)\leq1$, given by
\begin{equation}
\label{defOm1}
\BM =S_q^*- S_q^*(I-\la^*\la)^{-1} E_q\de_0^{-2}E_q^*.
\end{equation}
Furthermore, $B_{\nabla}= (I-\la^*\la)^{-1}\la^*S_p^*T_\tht E_k$, which maps $\BC^k$ into  $\ell_+^2(\BC^q)$, and $\Delta_0$ and $\Delta_1$ are the positive definite matrices  given by
\begin{align}
\Delta_0^2 &=E_q^*(I-\la^* \la)^{-1}E_q \nn\\
\Delta_1^2 &=I_k+E_k^* T_\tht^*  {S_p} \la (I-\la^* \la)^{-1}\la^*  {S_p^*} T_\tht E_k. \label{CLTDeltas1}
\end{align}
Moreover, the $(k+m)\times(m+p)$ coefficient matrix   $\Phi$ defined by
\begin{equation}
\label{defPhiA}
\Phi=\mat{cc}{\Phi_{11}&\Phi_{12}\\ \Phi_{21}& \Phi_{22}},
\end{equation}
with $\Phi_{11},$ $\Phi_{12}$, $\Phi_{21}$ and $\Phi_{22}$  defined above, is inner.
\end{thm}

It is useful to first prove some preliminary results.

The description of intertwining liftings in Theorem VI.6.1 in \cite{FFGK98} is with respect to the {\em Sz-Nagy-Sch\"affer isometric lifting} $U'_{NS}$ of $T'$, which is given by
\begin{equation} \label{lifting3}
U_{NS}'=\begin{bmatrix}
   T'   &    0\\
   E'D' &  S_{\sD'}
\end{bmatrix}\mbox{ on }\mat{c}{\sH'\\ \ell^2_+(\sD')}.
\end{equation}
Here $D'$ is the defect operator defined by $T'$, and $\sD'$ is the corresponding defect space, i.e., $D'=(I-T'^*T')^\half$ and $\sD'$ is the closure of $\im D'$. Furthermore, $E':\sD' \to \ell^2_+(\sD')$ is the canonical embedding defined by $(E'd')=(d',0,0,\ldots)$ for each $d'\in \sD'$. It is well known that $U'_{NS}$ is a minimal isometric lifting of $T'$. Since $S_p$ is assumed to be a minimal isometric lifting, there exists a unique unitary operator $\Psi_0$  mapping $\ell^2_+(\sD')$ onto $\im T_\tht=\ell^2_+(\BC^p)\ominus \sH'$ such that
\begin{equation}
\label{defPsi0a}
\begin{bmatrix}
I_{\sH'}  &    0\\
0 &  \Psi_0
\end{bmatrix}\begin{bmatrix}
   T'   &    0\\
   E'D' &  S_{\sD'}
\end{bmatrix}=\begin{bmatrix}
   T'   &    0\\
     W &  Z
\end{bmatrix}\begin{bmatrix}
I_{\sH'}  &    0\\
0 &  \Psi_0
\end{bmatrix}.
\end{equation}
The next lemma provides a description of the unitary operator $\Psi_0$.

\begin{lem} \label{lemPsiXi}
{Assume $S_ p$ is a minimal isometric lifting of $T^\prime$.} Let $\Psi_0$ be the unitary operator defined by \eqref{defPsi0a}, and let $\Xi$ be the unitary operator defined by
\begin{equation}\label{defXi}
\Xi: \ell^2_+(\BC^k)\to \ell^2_+(\BC^p)\ominus\sH', \quad \Xi g=T_\tht g\quad (g\in \ell^2_+(\BC^k)).
\end{equation}
Then there exists a unitary operator $N_0$ from $\sD'$ onto $\BC^k$ such that $\Psi_0=\Xi T_{N_0}$, with $T_{N_0}$ the diagonal Toeplitz operator defined by the constant function with value $N_0$, i.e.,
\begin{equation}
\label{prN0}
\Psi_0 f =T_\Theta T_{N_0}f \quad (f\in\ell^2_+(\sD')).
\end{equation}
Moreover,
\begin{itemize}
\item[\textup{(i)}] the matrix $N_0$ is uniquely determined by the identity
\begin{equation}\label{defN0}
N_0 D'=E_k^*\Xi^*W;
\end{equation}
\item[\textup{(ii)}]  the operator $W$ in \eqref{lifting2} is given by $W=T_\tht E_k N_0 D'$.
\end{itemize}
\end{lem}

\begin{proof}[\bf Proof.]
From the definition of $\Xi$  and the fact that $\tht$ is inner we see that  $T_\tht$ admits the following partitioning:
\[
T_\tht=\begin{bmatrix}0\\ \Xi\end{bmatrix}: \ell^2_+(\BC^k)\to
\begin{bmatrix}\sH'\\ \im T_\tht\end{bmatrix}.
\]
Since $S_pT_\tht=T_\tht S_k$, this implies that
\[
\begin{bmatrix} I_{\sH'}  &    0\\ 0 &  \Xi^*  \end{bmatrix}
\begin{bmatrix} T'   &    0\\  W &  Z \end{bmatrix}
=\begin{bmatrix} T'   &    0\\  \Xi^*W &  S_k \end{bmatrix}
\begin{bmatrix} I_{\sH'}  &    0\\ 0 &  \Xi^*  \end{bmatrix}.
\]
But then  \eqref{defPsi0a} yields
\[
\begin{bmatrix} I_{\sH'}  &    0\\ 0 & \Xi^* \Psi_0  \end{bmatrix}
\begin{bmatrix} T'   &    0\\ E'D' &  S_{\sD'} \end{bmatrix}
=\begin{bmatrix} T'   &    0\\  \Xi^*W &  S_k \end{bmatrix}
\begin{bmatrix} I_{\sH'}  &    0\\ 0 &  \Xi^* \Psi_0 \end{bmatrix}.
\]
In particular, $ (\Xi^* \Psi_0)S_{\sD'}= S_k (\Xi^* \Psi_0)$. Since the operator $\Xi^* \Psi_0$ is unitary, the latter intertwining relation implies that $\Xi^* \Psi_0$ is a block diagonal Toeplitz operator $T_{N_0}=\diag (N_0, N_0, \ldots)$,
where  $N_0$ is a unitary operator from  $\sD'$ onto $\BC^k$.

The identity  $T_{N_0}=\Xi^* \Psi_0$ and the fact that $\Xi$ is unitary imply that $ \Xi T_{N_0}=\Psi_0$. Using the definition of $\Xi$ in \eqref{defXi}  the latter identity yields \eqref{prN0}. Finally, from $T_{N_0}E'D'=\Xi^* \Psi_0E'D'=\Xi^*W$ we obtain \eqref{defN0}.
\end{proof}

\begin{proof}[\bf Proof of Theorem \ref{cltstrict}.]
The characterization of all solutions in \eqref{RedSolCLT} follows by applying Theorem VI.6.1 from \cite{FFGK98} to the commutant lifting data described above. Note that $\|\la\|<\g=1$ implies $\la$ is a strict contraction. Directly applying the formulas from \cite{FFGK98}, using $A=\Lambda$, $T=S_q$ and $\Pi_0=E_q^*$ and multiplying with $\Theta(z)N_0$ on the right, as noted in Lemma \ref{lemPsiXi}, we obtain that the functions $X$ in $H^\iy_{p\ts q}$ satisfying \eqref{clt2a} are given by \eqref{RedSolCLT} with
\begin{align}
\notag \Phi_{11}(z)
&=-z\Delta_0^{-1}E_q^*(I-z\BM)^{-1}(I-\la^*\la)^{-1}\Lambda^*D'\wtil{\Delta}_1^{-1}N_0^*\\
\label{PhisCLT2}
\Phi_{12}(z) &=\Delta_0^{-1}E_q^*(I-z\BM)^{-1}E_q\\
\notag \Phi_{21}(z)
&=\Theta(z)N_0(\wtil{\Delta}_1^2-D'\Lambda(I-z\BM)^{-1}(I-\la^*\la)^{-1}\Lambda^*D')\wtil{\Delta}_1^{-1}N_0^*\\
\notag \Phi_{22}(z)
&=E_p^*(I-zS_p^*)^{-1}\Lambda E_q+\Theta(z)N_0 D'\Lambda
\BM(I-z\BM)^{-1}E_q.
\end{align}
where $\Delta_0$ (in \cite{FFGK98} denoted by $N$) is as in \eqref{CLTDeltas1} and $\BM$ and $\wtil{\Delta}_1$  (in \cite{FFGK98} denoted by  {$T_A^*$} and $N_1$, respectively) are given by
\begin{equation}\label{CLTomDe1}
\BM =(I-S_q^* \la^* \la S_q)^{-1} S_q^* (I-\la^* \la)\ands
\wtil{\Delta}_1^2 =I_k+D'\Lambda D_\la^{-2}\Lambda^* D'.
\end{equation}
Here we multiplied the formulas in \cite{FFGK98} for $\Phi_{11}$ and $\Phi_{21}$ with the unitary operator $N_0^*:\BC^k\to \sD'$ from Lemma \ref{lemPsiXi}, so that the free parameter function $Y$ maps into the right space.

Using the fact that $N_0$ is a unitary operator satisfying \eqref{defN0}, it is obvious that $N_0 \wtil{\Delta}_1^2=\Delta_1^2 N_0$. Then, also $N_0 \wtil{\Delta}_1=\Delta_1 N_0$ and $N_0 \wtil{\Delta}_1^{-1}=\Delta_1^{-1} N_0$. It remains to show that the formulas for $\BM$ in \eqref{CLTomDe1} and \eqref{CLTDeltas1} coincide. Indeed, once this fact is established, it easily follows from the intertwining relations for $\wtil{\Delta}_1$ and $\Delta_1$, together with \eqref{defN0}, that the functions $\Phi_{ij}$ in \eqref{PhisCLT2} are also given by \eqref{PhisCLT1}.

To see that the two formulas for $\BM$ coincide, note that
\begin{align*}
\BM&=(I-S_q^* \la^* \la S_q)^{-1} S_q^* (I-\la^* \la)\\
&=S_q^*(I-\la^*\la S_qS_q^*)^{-1} (I-\la^*\la)\\
&=S_q^*(I-\la^*\la (I-E_qE_q^*))^{-1} (I-\la^*\la)\\
&=S_q^*((I-\la^*\la)+\la^*\la E_qE_q^*)^{-1}(I-\la^*\la)\\
&=S_q^*(I+(I-\la^*\la)^{-1}\la^*\la E_qE_q^*)^{-1}.
\end{align*}
Now set
\[
A=I,\quad B=E_q^*,\quad C=(I-\la^*\la)^{-1}\la^*\la E_q,\quad D=I.
\]
Since $I+(I-\la^*\la)^{-1}\la^*\la E_qE_q^*=D+CA^{-1}B$ is invertible, so is
\begin{align*}
A^\times &:=A+B D^{-1}C=I+E_q^*(I-\la^*\la)^{-1}\la^*\la E_q\\
&=E_q^*(I+(I-\la^*\la)^{-1}\la^*\la)E_q
=E_q^*(I-\la^*\la)^{-1}E_q
=\de_0^2.
\end{align*}
By standard inversion formulas, cf., \cite{BGK79}, we obtain that
\begin{align*}
\BM &=S_q^*(D^{-1}-D^{-1}C(A^\times)^{-1}BD^{-1})\notag\\
&=S_q^*(I-(I-\la^*\la)^{-1}\la^*\la E_q\de_0^{-2}E_q^*)\notag\\
&=S_q^*- S_q^*(I-\la^*\la)^{-1}\left(I-(I-\la^*\la)\right) E_q\de_0^{-2}E_q^*\notag \\
&=S_q^*- S_q^*(I-\la^*\la)^{-1}E_q\de_0^{-2}E_q^* +S_q^*E_q\de_0^{-2}E_q^*\notag\\
&=S_q^*- S_q^*(I-\la^*\la)^{-1}E_q\de_0^{-2}E_q^*.
\end{align*}
Here we used that $S_q^*E_k=0$.  The latter identity implies $E_k^*S_q=0$, and hence $\BM S_q=I$.
Hence $\BM$ is given  by \eqref{defOm1}. Therefore
\begin{equation}\label{om2}
\BM = (I-S_q^* \la^* \la S_q)^{-1} S_q^* (I-\la^* \la) = S_q^*- S_q^*(I-\la^*\la)^{-1}E_q\de_0^{-2}E_q^*.
\end{equation}
\end{proof}

As in \cite{FtHK-IEOT} we shall need the following functions:
\begin{align}
U(z) &= E_p^* \left(I - z S_p^*\right)^{-1} \la\left(I - \la^* \la  \right)^{-1}  E_q, \label{defU} \\
V(z) &= E_q^* \left(I - z S_q^*\right)^{-1}  \left(I - \la^* \la  \right)^{-1} E_q.
\label{defV}
\end{align}
As mentioned in Theorem 2.1 in \cite{FtHK-IEOT},  $\det V(z)\not = 0$ for $| z|<1$, the function $V^{-1}$ belongs to  $H_{q\times q}^\infty$ and is  an outer function.

\begin{prop}
Let $\Phi_{12}$ and $\Phi_{22}$ be as in \eqref{PhisCLT1}, and let $U$ and $V$ be given by \eqref{defU} and \eqref{defV}, respectively. Then
\begin{equation}
\label{2Phis}
\Phi_{12}(z) =\de_0 V(z)^{-1} \ands  \Phi_{22}(z)  =U(z)V(z)^{-1}\quad(z\in \BD).
\end{equation}
\end{prop}

\begin{proof}[\bf Proof]
%
%
First we prove the first identity in \eqref{2Phis}. From the definition of  $\Phi_{12}$ in  \eqref{PhisCLT1} it is clear that
\[
\Phi_{12}(z)  = \de_0^{-1}+z \de_0^{-1}E_q^* \left(I-z\BM\right)^{-1}\BM E_q.
\]
Using \cite[Theorem 2.1]{BGKR08}, it follows that  in a neighborhood of zero we  have
\[
\Phi_{12}(z) ^{-1}=\de_0 -zE_q^* \left(I-z\BM^\ts  \right)^{-1}\BM E_q\de_0.
\]
This with \eqref{om2} yields
\begin{align}
\BM^\ts  &=\BM -(\BM E_q)\de_0 (\de_0 ^{-1}E_q^*)=\BM-\BM E_qE_q^*\label{omts1}\\
&=\BM S_qS_q^* =(I-S_q^* \la^* \la S_q)^{-1} S_q^* (I-\la^* \la)S_qS_q^*
=S_q^*. \label{omts2}
\end{align}
Using \eqref{omts2} it follows that
\begin{equation}
\Phi_{12}(z)^{-1}=\de_0 -zE_q^* \left(I-zS_q^* \right)^{-1}\BM E_q\de_0, \quad z\in \BD.\label{Phi12min}
\end{equation}
Next, note that
\begin{align*}
\BM E_q\de_0^2&=\left( S_q^*- S_q^*(I-\la^*\la)^{-1}E_q\de_0^{-2}E_q^* \right)E_q\de_0^2\\
&=- S_q^*(I-\la^*\la)^{-1}E_q\de_0^{-2}E_q^* E_q\de_0^2\\
&=- S_q^*(I-\la^*\la)^{-1}E_q.
\end{align*}
Hence
\begin{align*}
\Phi_{12}(z)^{-1}\de_0&=\de_0^2-zE_q^* \left(I-zS_q^* \right)^{-1}\BM E_q\de_0^2\\
&=\de_0^2+zE_q^* \left(I-zS_q^* \right)^{-1} S_q^*(I-\la^*\la)^{-1}E_q\\
&=\de_0^2+ E_q^* \left(I-zS_q^* \right)^{-1} \left(I- (I-zS_q^*)\right)(I-\la^*\la)^{-1}E_q\\
&=\de_0^2+ E_q^* \left(I-zS_q^* \right)^{-1}(I-\la^*\la)^{-1}E_q-E_q^* (I-\la^*\la)^{-1}E_q\\
&=E_q^* \left(I-zS_q^* \right)^{-1}(I-\la^*\la)^{-1}E_q=V(z).
\end{align*}
This proves the first identity in \eqref{2Phis}.

To prove  the second identity in \eqref{2Phis}, note that $\Phi_{22}$ is the so-called central solution, i.e, the solution that one obtains if the free parameter $Y$ in \eqref{RedSolCLT} is taken to be zero.  But then  \cite[Theorem IV.7.1]{FFGK98}  tells us that  $\Phi_{22}$ is the maximum entropy solution and we can apply \cite[Propositon 3.1]{FtHK-IEOT} to show that the second identity in \eqref{2Phis} holds true. For the sake of completeness we   also give a direct proof.

We take $\Phi_{22}$  as in \eqref{PhisCLT2}. This formula can be rewritten as
\begin{align*}
\Phi_{22}(z)&=E_p^*(I-zS_p^*)^{-1}\Big(\la (I-z\BM) +T_\tht E_k N_0 D' \la
\BM\Big)(I-z\BM)^{-1}E_q\\
&=E_p^*(I-zS_p^*)^{-1}\Big(\la (I-z\BM)+W\la \BM \Big)(I-z\BM)^{-1}E_q.
\end{align*}
Here we used the identity $\tht(z)=E_p^*(I-zS_p^*)^{-1}T_\tht E_k$  and item (ii) in Lemma \ref{lemPsiXi}. Put $M(z)= \la (I-z\BM)+W\la \BM$. This operator function admits the following partitioning:
\[
M(z)=\begin{bmatrix} \la (I-z\BM) \\[.2cm]  W\la \BM\end{bmatrix}: \ell_+^2(\BC^q)\to \begin{bmatrix} \sH'\\[.2cm]  \im T_\tht \la \BM\end{bmatrix}.
\]
Using this partitioning, formula  \eqref{lifting2}, the intertwining relation $T'\la =\la S_q$, and the fact that $\BM S_q=I$,  we see that
\begin{align*}
M(z)S_q&=\begin{bmatrix} \la S_q-z\la \\[.2cm]  W\la\end{bmatrix}=
\begin{bmatrix} T'  \la\\[.2cm]  W\la\end{bmatrix}-z
\begin{bmatrix}   \la\\[.2cm]  0\end{bmatrix}\\
&= S_p\begin{bmatrix}   \la\\[.2cm]  0\end{bmatrix}- z\begin{bmatrix}   \la\\[.2cm]  0\end{bmatrix}=(I-zS_p^*)S_p \begin{bmatrix}   \la\\[.2cm]  0\end{bmatrix}.
\end{align*}
If follows that
\begin{equation}
\label{auxr1}
E_p^*(I-zS_p^*)^{-1}M(z)S_q=E_p^*(I-zS_p^*)^{-1} (I-zS_p^*)S_p\begin{bmatrix}   \la\\[.2cm]  0\end{bmatrix}=0.
\end{equation}
Applying this to our formula for $\Phi_{22}$ we obtain
\begin{align*}
\Phi_{22}(z)&=E_p^*(I-zS_p^*)^{-1}M(z)(I-z\BM)^{-1}E_q\\
&=E_p^*(I-zS_p^*)^{-1}M(z)(E_qE_q^*+S_qS_q^*)(I-z\BM)^{-1}E_q\\
&=E_p^*(I-zS_p^*)^{-1}M(z)E_qE_q^*(I-z\BM)^{-1}E_q.
\end{align*}
Using
the definition of $\Phi_{12}$ in \eqref{PhisCLT2}, and the definition of  $\de_0$ in \eqref{CLTDeltas1}, we see that
\[
E_q^*(I-z\BM)^{-1}E_q=\de_0\Phi_{12}=\de_0^2V(z)^{-1}=E_q^*(I-\la^*\la)^{-1}E_qV(z)^{-1}.
\]
Together with $E_qE_q^*=I-S_qS_q^*$ and the identity \eqref{auxr1} the previous identity yields
\[
\Phi_{22}(z)=E_p^*(I-zS_p^*)^{-1}\begin{bmatrix} \la (I-z\BM) \\[.2cm]  W\la \BM\end{bmatrix}(I-\la^*\la)^{-1}E_qV(z)^{-1}.
\]
Finally, using the formula for $\BM$ given by the left hand side of \eqref{om2} and $S_q^*E_q=0$,  we see that
\[
\BM (I-\la^*\la)^{-1}E_q=(I-S_q^*\la^*\la S_q)^{-1}S_q^*Eq=0.
\]
Hence the above formula for $\Phi_{22}$ simplifies to
\[
\Phi_{22}(z)=E_p^*(I-zS_p^*)^{-1}\begin{bmatrix} \la \\[.2cm]  0 \end{bmatrix}(I-\la^*\la)^{-1}E_qV(z)^{-1}.
\]
Using the definition of $U$ in \eqref{defU}, this yields the second identity in \eqref{2Phis}. \end{proof}

The following result  in the  analogue of Theorem \ref{cltstrict} with the Redheffer representation of all solution  \eqref{RedSolCLT} being replaced by a linear fractional map.

\begin{thm}\label{cltstrict2}
Assume $S_ p$ is a minimal isometric lifting of $T^\prime$, and let $\la$ be a strict contraction mapping $\ell_+^2(\BC^q)$ into $\sH^\prime\subset \ell_+^2(\BC^p)$ satisfying the intertwining relation $T^\prime \la= \la S_q$. Then   all functions $X$ in $H_{p\times q}^\infty$ satisfying
\begin{equation}\label{clt2b}
\la = P_{\sH^\prime} T_X
\quad \mbox{and}\quad \|X\|_\infty \leq 1
\end{equation}
are given by
\begin{equation}\label{LFTSolCLT}
X(z)=\Big(\Up_{12}(z)+\Up_{11}(z)Y(z)\Big) \Big(\Up_{22}(z)+\Up_{21}(z)Y(z)\Big)^{-1}, \quad |z|<1.
\end{equation}
Here $Y$ is an arbitrary function in $H^\iy_{k\ts q}$ with $\|Y\|_\infty\leq 1$, and
\begin{align}
\Up_{11}(z)&=zE_p^*(I-zS_p^*)^{-1}\la E_q E_q^*(I-zS_q^*)^{-1}B_{\nabla}\de_1^{-1}+\nn \\
&\hspace{1cm}+\tht(z)\de_1^{-1}-z\tht(z)E_k^*T_\tht^*S_p\la (I-zS_q^*)^{-1}S_q^*B_{\nabla}\de_1^{-1},\nn \\
\Up_{21}(z)&=zE_q^*(I-zS_q^*)^{-1}B_\nabla \de_1^{-1}  \label{LFUps} \\
\Up_{12}(z)&=U(z) \de_0^{-1},\nn \\
\Up_{22}(z)&=V(z) \de_0^{-1}.\nn
\end{align}
Here $B_{\nabla}= (I-\la^*\la)^{-1}\la^*S_p^*T_\tht E_k$, the functions $U$ and $V$ are given by \eqref{defU} and \eqref{defV}, respectively,  and $\Delta_0$ and $\Delta_1$ are the positive definite matrices  given by
\begin{align}
\Delta_0^2 &=E_q^*(I-\la^* \la)^{-1}E_q \nn\\
\Delta_1^2 &=I_k+E_k^* T_\tht^* S_p \la (I-\la^* \la)^{-1}\la^* S_p^* T_\tht E_k. \label{Deltas2}
\end{align}
Moreover, the $(p+k)\times(q+p)$ coefficient matrix $\Up$ defined  by
\[
\Up=\mat{cc}{\Up_{11}&\Up_{12}\\ \Up_{21}& \Up_{22}},
\]
with $\Up_{11},$ $\Up_{12}$, $\Up_{21}$ and $\Up_{22}$ as above, is $J_1, J_2$-inner, where    $J_1$ and $J_2$ are given by $J_1=\diag(I_p, -I_q)$, and $J_2=\diag(I_{k}, -I_q)$.
\end{thm}

\begin{proof}[\bf Proof.]
The fact that $\Phi_{12}(z)$ is invertible for each $z\in\BD$, with an analytic inverse, implies that we can apply the Potapov-Ginzburg transform pointwise, cf., Section 2.5 in \cite{AD08}, defining analytic matrix valued functions $\Up_{ij}$, $i,j=1,2$, on $\BD$ via
\begin{equation}\label{PGtrans1A}
\begin{aligned}
&\Up_{11}=\Phi_{21}-\Phi_{22}\Phi_{12}^{-1}\Phi_{11},\quad
\Up_{12}=\Phi_{22}\Phi_{12}^{-1},\\
&\Up_{21}=-\Phi_{12}^{-1}\Phi_{11},\quad
\Up_{22}=\Phi_{12}^{-1}.
\end{aligned}
\end{equation}
Following \cite{AD08}, we obtain that the identity
\[
\Phi_{22}+\Phi_{21}Y(I-\Phi_{11}Y)^{-1}\Phi_{12}
=(\Up_{12}+\Up_{11}Y)(\Up_{22}+\Up_{21}Y)^{-1}
\]
holds point wise on $\BD$ for any function $Y$ in $H^\infty_{k\ts q}$ with $\|Y\|_\infty\leq 1$. Moreover, since $\Phi$ in \eqref{defPhiA} is inner, we obtain that the coefficient matrix
\begin{equation}\label{UpBlockInfA}
\Up=\mat{cc}{\Up_{11}&\Up_{12}\\ \Up_{21}& \Up_{22}},
\end{equation}
is $J_1, J_2$-inner, where    $J_1=\diag(I_p, -I_q)$, and $J_2=\diag(I_{k}, -I_q)$, that is, for almost any  $z\in\BT$ we have $\Up(z)^*J_1 \Up(z)=J_2$.

From the results in the previous paragraph we conclude that in order to prove the theorem it suffices to show that the functions $\Up_{ij}$, $1\leq i, j \leq 2$, defined in  \eqref{PGtrans1A},  are also given by the right  hands of the formulas   in \eqref{LFUps}. For  $\Up_{12}$ and $\Up_{22}$ this follows directly from the two identities in \eqref{2Phis}.   So it remains to consider the functions $\Up_{11}$ and $\Up_{21}$.

We begin with $\Up_{21}$. Using the definition of  $\Up_{21}$ in \eqref{PGtrans1A},
the identity \eqref{Phi12min},  and the first identity in \eqref{PhisCLT1}, we see that
\begin{align*}
\Up_{21}(z)&=-\Phi_{12}(z)^{-1}\Phi_{11}(z)\\
&=-\Big(I_q-zE_q^*(I-zS_q^*)^{-1}\BM E_q\Big)\de_0\ts\\
&\hspace{2cm}\ts \Big(-z\de_0^{-1}E_q^*(I-z\BM)^{-1}B_\nabla\de_1^{-1}\Big)\\
&=z E_q^*(I-z\BM)^{-1}B_\nabla\de_1^{-1}+\\
&\hspace{2cm}-zE_q^*(I-zS_q^*)^{-1}\Big(z\BM E_qE_q^*\Big)(I-z\BM)^{-1}B_\nabla\de_1^{-1}.
\end{align*}
From \eqref{omts1} and  \eqref{omts2} we see that
\begin{equation}
\label{identomS}
\BM - S_q^*=\BM E_qE_q^*.
\end{equation}
Using the latter identity we obtain
\begin{align*}
&(I-zS_q^*)^{-1}\Big(z\BM E_qE_q^*\Big)(I-z\BM)^{-1}=\\
&\hspace{1cm}=(I-zS_q^*)^{-1}\Big(z\BM -zS_q^*\Big)(I-z\BM)^{-1}\\
&\hspace{1cm}=(I-zS_q^*)^{-1}\Big((I-zS_q^*)-(I-z\BM)\Big)(I-z\BM)^{-1}\\
&\hspace{1cm}=(I-z\BM)^{-1}-(I-zS_q^*)^{-1}.
\end{align*}
It follows that
\begin{align*}
\Up_{21}(z)&=zE_q^*(I-z\BM)^{-1}B_\nabla\de_1^{-1}+\\
&\hspace{.7cm}-zE_q^*(I-z\BM)^{-1}B_\nabla\de_1^{-1}+zE_q^*(I-zS_q^*)^{-1}B_\nabla\de_1^{-1}\\
&=zE_q^*(I-zS_q^*)^{-1}B_\nabla\de_1^{-1}.
\end{align*}
This proves  the second identity in \eqref{Deltas2}.

Next we deal with $\Up_{11}$. According to \eqref{PGtrans1A},  we have
\[
\Up_{11}(z)=\Phi_{21}(z)-\Phi_{22}(z)\Phi_{12}(z)^{-1}\Phi_{11}(z)=\Phi_{21}+\Phi_{22}(z)\Up_{21}(z).
\]
We first compute $\Phi_{22}\Up_{21}$ using the first identity in \eqref{PhisCLT1} and the second in \eqref{LFUps}. This yields
\begin{align*}
\Phi_{22}(z)\Up_{21}(z)&=z\Phi_{22}(z)E_q^*(I-zS_q^*)^{-1}B_\nabla\de_1^{-1}\\
&=A(z)+B(z),
\end{align*}
where
\begin{align*}
A(z)&=zE_p^*(I-zS_p^*)^{-1}\la E_qE_q^*(I-zS_q^*)^{-1}B_\nabla \de_1^{-1},\\[.2cm]
B(z)&=z\tht(z)E_k^*T_\tht^*S_p\la\BM (I-z\BM)^{-1}E_qE_q^*(I-zS_q^*)^{-1}B_\nabla \de_1^{-1}.
\end{align*}
Again using the identity in \eqref{identomS} we obtain
\[
z\BM(I-z\BM)^{-1}E_qE_q^*(I-zS_q^*)^{-1}=(I-z\BM)^{-1}-(I-zS_q^*)^{-1}.
\]
This yields
\begin{align*}
B(z)&=\tht(z)E_k^*T_\tht^*S_p\la(I-z\BM)^{-1}B_\nabla \de_1^{-1}+\\
&\hspace{2cm}-\tht(z)E_k^*T_\tht^*S_p\la(I-zS_q^*)^{-1}B_\nabla \de_1^{-1}.
\end{align*}
Recall that  $\Phi_{21}$ is given by the third identity in \eqref{PhisCLT1}. If follows that
\[
\Phi_{21}(z)+B(z)=\tht(z)\de_1-\tht(z)E_k^*T_\tht^*S_p\la(I-zS_q^*)^{-1}B_\nabla \de_1^{-1}.
\]
Hence
\begin{align}
\Up_{11}(z)&=\Phi_{21}(z) +\Phi_{22}(z)\Up_{21}(z)=\Phi_{21}(z) +A(z)+B(z)\nn\\
&=zE_p^*(I-zS_p^*)^{-1}\la E_qE_q^*(I-zS_q^*)^{-1}B_\nabla \de_1^{-1}+\nn \\[.2cm]
&\hspace{1.4cm} +\tht(z)\de_1-\tht(z)E_k^*T_\tht^*S_p\la(I-zS_q^*)^{-1}B_\nabla \de_1^{-1}.\label{LFUp11a}
\end{align}
To get the first identity in \eqref{LFUps} we have to do one additional step. Note that $I=(I-zS_q^*)-zS_q^*$. Hence
\begin{align*}
& E_k^*T_\tht^*S_p\la(I-zS_q^*)^{-1}B_\nabla \de_1^{-1}=\\
&\hspace{.5cm}=  E_k^*T_\tht^*S_p\la(I-zS_q^*)^{-1}\Big((I-zS_q^*)-zS_q^*\Big)B_\nabla \de_1^{-1}\\
&\hspace{.5cm}=  E_k^*T_\tht^*S_p\la B_\nabla \de_1^{-1}+\\
&\hspace{2cm}-z E_k^*T_\tht^*S_p\la(I-zS_q^*)^{-1}S_q^*B_\nabla \de_1^{-1}.
\end{align*}
Next, using the definitions of  $B_\nabla$ and $\de_1$  in Theorem \ref{cltstrict2}, we have
\begin{align*}
E_k^*T_\tht^*S_p\la B_\nabla \de_1^{-1}&=E_k^*T_\tht^*S_p \la   (I-\la^*\la)^{-1}\la^*S_p^*T_\tht E_k\de_1^{-1}\\
&=(\de_1^2-I_k)\de_1^{-1}=\de_1-\de_1^{-1}.
\end{align*}
It follows that
\begin{align*}
&\tht(z)E_k^*T_\tht^*S_p\la(I-zS_q^*)^{-1}B_\nabla \de_1^{-1}=\\
&\hspace{1cm}=\tht(z)\de_1
-\tht(z)\de_1^{-1}-z\tht(z)E_k^*T_\tht^*S_p\la(I-zS_q^*)^{-1}S_q^*B_\nabla \de_1^{-1}.
\end{align*}
Using the latter identity in \eqref{LFUp11a}, we obtain the first identity in \eqref{LFUps}.
\end{proof}

\paragraph{Comment on the Toeplitz corona problem.} The Toeplitz corona problem can be reduced to the special case of the Leech problem where $q=m$ and $K$ is identically equal to $I_m$. In that case the solvability condition is that  $T_GT_G^*\geq I$, and thus $T_GT_G^*$ is strictly positive. Being a special case of the Leech problem, the Toeplitz corona problem can be formulated as a commutant lifting problem of the form considered in this section, where $\Lambda=T_G^*(T_GT_G^*)^{-1}$ viewed as an operator mapping $\ell^2_+(\BC^m)$ into $\sH'=\im T_G^*$.
Note that in this case $\Lambda$ is an invertible contraction.

\begin{prop}\label{P:coronaCLT}
Let $\Lambda$ be an invertible contraction mapping $\ell^2_+(\BC^q)$ into $\sH'=\kr T_\Theta^*$, with $\Theta\in H^\iy_{p\ts k}$ an inner function, and assume that $\la$ intertwines $S_q$ with the compression of $S_p$ to $\sH'$. Then there exists  a function $G\in H^\iy_{m\ts p}$ such that   $T_G$  is right invertible, the space $\sH'=\im T_G^*$,  and $\Lambda = T_G^*(T_GT_G^*)^{-1}$ viewed as an operator mapping $\ell^2_+(\BC^m)$ into $\sH'$. In fact,  $T_G=\la^{-1}\Pi'$,
where  $\Pi': \ell^2_+(\BC^p)\to\sH'$ denotes the orthogonal projection onto $\sH'$.
\end{prop}

\begin{proof}[\bf Proof] Put $T:=\Lambda^{-1}\Pi'$. It suffices to show that $T$ is a Toeplitz operator since clearly $T$ is left invertible, $\im T^*=\sH'$, and
\[
T^*(TT^*)^{-1}
=\Pi'^*\Lambda^{-*} (\Lambda^{-1}\Pi'\Pi'^*\Lambda^{-*})^{-1}
=\Pi'^*\Lambda^{-*} (\Lambda^{-1}\Lambda^{-*})^{-1}=\Pi'^* \Lambda.
\]
To see that $T$ is Toeplitz, note that $T'\Lambda=\Lambda S_m$ implies $\Lambda^{-1}T'=S_m\Lambda^{-1}$. Using that $S_p$ is an isometric lifting
of $T'$, we find
\[
S_m T= S_m \Lambda^{-1}\Pi'=\Lambda^{-1}T'\Pi'=\Lambda^{-1}\Pi'S_p=T S_p,
\]
which proves our claim.
\end{proof}



\end{document}